\documentclass[a4paper]{article}
\usepackage{amsmath,amsthm,amssymb,amscd,graphicx, enumerate}

\numberwithin{equation}{section}

\newtheorem{thm}{Theorem}[section]
\newtheorem{lem}[thm]{Lemma}
\newtheorem{prop}[thm]{Proposition}

\newtheorem{remark}[thm]{Remark}

\theoremstyle{definition}
\newtheorem{definition}[thm]{Definition}
\newtheorem{example}[thm]{Example}

\newenvironment{rem}{\begin{remark}\rm}{\end{remark}}

\def\address#1#2{\begingroup
\noindent\parbox[t]{12cm}{%
\small{\scshape\ignorespaces#1}\par\vskip1ex
\noindent\small{\itshape E-mail address}%
\/: #2\par\vskip4ex}\hfill%
\endgroup}%
\makeatother

\title{Some associative submanifolds of the squashed 7-sphere}
\author{Kotaro Kawai
\footnote{The author is supported by Grant-in-Aid for JSPS fellows (26-7067).}}
\date{}

\begin{document}

\maketitle

\begin{abstract}
The squashed 7-sphere $S^{7}$ is a 7-sphere 
with an Einstein metric given by the canonical variation 
and its cone $\mathbb{R}^{8} - \{ 0 \}$ has full holonomy ${\rm Spin}(7)$. 
There is a canonical calibrating 4-form $\Phi$ on $\mathbb{R}^{8} - \{ 0 \}$. 
A minimal 3-submanifold in $S^{7}$
is called associative if its cone is calibrated by $\Phi$.

In this paper, we classify two types of 
fundamental associative submanifolds in the squashed $S^{7}$.  
One is obtained by the intersection with a 4-plane 
and 
the other is homogeneous. 
Then 
we study their infinitesimal associative deformations and 
explicitly show that all of them are integrable. 
\end{abstract}

\section{Introduction}

A Riemannian 7-manifold $(Y, g)$ is 
called a nearly parallel $G_{2}$-manifold  
if its cone $(C(Y), \overline{g}) = (\mathbb{R}_{>0} \times Y, dr^{2} + r^{2} g)$ 
has holonomy contained in ${\rm Spin(7)}$. 
The existence of such a structure 
is equivalent to that of a spin structure with 
a real Killing spinor (\cite{Bar}), 
which is also used 
in supergravity and superstring theory in physics. 
There is a canonical calibrating 4-form $\Phi$ on $C(Y)$. 
A 3-submanifold $M$ in $Y$ is called associative if 
its cone $C(M)$ is Cayley, 
i.e. it is calibrated by $\Phi$. 

By definition, Sasaki-Einstein manifolds, especially 3-Sasakian manifolds, 
admit nearly parallel $G_{2}$-structures.
Moreover, every compact 3-Sasakian 7-manifold admits 
a second nearly parallel $G_{2}$-structure whose cone metric 
has full holonomy ${\rm Spin}(7)$(\cite{FKMS}). 
The 7-sphere $S^{7}$ with this second nearly parallel $G_{2}$-structure is called 
the squashed $S^{7}$.

Associative submanifolds in the standard $S^{7}$ were studied by Lotay \cite{Lotay3}. 
In this paper, we study some fundamental associative submanifolds in the squashed $S^{7}$ 
and compare the properties. 

First, we find some fundamental examples of  associative submanifolds in the squashed $S^{7}$. 
Fibers of the Hopf fibration $\pi: S^{7} \rightarrow S^{4}$ are associative. 
More generally, the Hopf lifts of $I_{1}^{'}$-holomorphic curves in $\mathbb{C}P^{3}$ 
are also associative in the squashed $S^{7}$(Proposition \ref{example asso psehol}),
where $I_{1}^{'}$ is an almost complex structure on $\mathbb{C}P^{3}$ given by (\ref{almcpxstr CP3}).

Next, we classify associative submanifolds obtained by the intersection with a 4-plane. 
Note that 
the automorphism group of the squashed $S^{7}$ is 
${\rm Sp}(1) {\rm Sp}(2) = {\rm Sp}(1) \times {\rm Sp}(2)/ \{ \pm (1, 1) \}$
(Lemma \ref{auto sq S7}).

\begin{thm} \label{plane asso}
Let  $V \subset \mathbb{R}^{8} = \mathbb{C}^{4}$ be a 4-plane. 
Suppose that $V \cap S^{7}$ is associative in the squashed $S^{7}$. 
Then up to the ${\rm Sp}(1) {\rm Sp}(2)$-action, $V$ is either 
\begin{align*}
V_{1} &= \{ (z_{1}, z_{2}, 0, 0) \in \mathbb{C}^{4}; z_{1}, z_{2} \in \mathbb{C} \} 
\quad \mbox{ or } \quad 
V_{2} = \{ (z_{1}, 0, z_{3}, 0) \in \mathbb{C}^{4}; z_{1}, z_{3} \in \mathbb{C} \}.  
\end{align*}
In other words, 
the space $\mathcal{M}$ 
of 4-planes whose intersections with $S^{7}$ are associative 
is described as 
\begin{align*}
\mathcal{M} = {\rm Sp}(1) {\rm Sp}(2)/K_{1} \sqcup {\rm Sp}(1) {\rm Sp}(2)/K_{2}, 
\end{align*}
where
$
K_{1} = {\rm Sp}(1) ({\rm Sp}(1) \times {\rm Sp}(1)), 
$
and  
$
K_{2} = {\rm U}(1) {\rm U}(2). 
$
\end{thm}

\begin{rem}
We see that $\mathcal{M}$ consists of two connected components, 
while 
the corresponding space in the standard $S^{7}$ is a 
homogeneous space 
${\rm Spin}(7)/K$, 
where $K = {\rm SU}(2)^{3}/\mathbb{Z}_{2}$ (\cite{Harvey Lawson}).

Note that $V_{1}$ is a quaternionic plane in $\mathbb{C}^{4} = \mathbb{H}^{2}$
and 
$V_{2}$ arises from a horizontal $I_{1}$-curve of $\mathbb{C}P^{3}$ 
in the sense of Remark \ref{basic concrete example of asso}. 
Moreover, both $V_{j} \cap S^{7}$, where $j=1,2$, are totally geodesic submanifolds in the squashed $S^{7}$. 
Actually, we should classify totally geodesic associative submanifolds, 
but it would be difficult because the squashed $S^{7}$ 
is neither a space of the constant curvature nor a symmetric space. 
It is just a homogeneous space ${\rm Sp}(1) {\rm Sp}(2)/{\rm Sp}(1) {\rm Sp}(1)$. 
\end{rem}

Next, we classify homogeneous associative submanifolds. 

\begin{thm} \label{classification homog asso}
Let $A$ be a connected associative 3-fold in the squashed $S^{7} \subset \mathbb{C}^{4}$ which is the orbit
of a closed Lie subgroup of ${\rm Sp}(1) {\rm Sp}(2)$. 
Then, up to the ${\rm Sp}(1) {\rm Sp}(2)$-action, $A$ is one of the following. 
\begin{enumerate}
\item $L_{1} = V_{1} \cap S^{7}$, where $V_{1}$ is given in Theorem \ref{plane asso}, 
\item $L_{2} = V_{2} \cap S^{7}$, where $V_{2}$ is given in Theorem \ref{plane asso}, 
\item 
$A_{1} = T^{3} \cdot \frac{1}{2} {}^t\! (1, 1, 1, i) \cong T^{3}$, where the $T^{3}$-action is given by (\ref{T3 action}), 
\item
$A_{2} = {\rm SU}(2) \cdot {}^t\! (1, 0, 0, 0) \cong {\rm SU}(2)/ \mathbb{Z}_{3}$, 
where the ${\rm SU}(2)$-action is given by (\ref{irr SU2}), 
\item
$A_{3} = {\rm SU}(2) \cdot {}^t\! (0, 0, 1, 0) \cong {\rm SU}(2)$, 
where the ${\rm SU}(2)$-action is given by (\ref{irr SU2}).
\end{enumerate}
\end{thm}

\begin{rem}
Since $T^{3}$ in (\ref{T3 action}) and ${\rm SU}(2)$ in (\ref{irr SU2}) are contained in 
${\rm SU}(4) \subset {\rm Spin}(7)$
by an appropriate change of  coordinates, 
we obtain the similar orbits 
$A_{1}, A_{2}$, and $A_{3}$ as in the standard $S^{7}$ case (\cite{Lotay3}).
However, since $G_{2}$ is not contained in ${\rm Sp}(1) {\rm Sp}(2)$, 
there are no corresponding associative orbits in the squashed $S^{7}$ 
to Lagrangian (totally real) submanifolds in $S^{6}$ classified by \cite{Mashimo}.
\end{rem}

\begin{rem} \label{property homog asso}
The examples 
$A_{1}, A_{2}$, and $A_{3}$ are Hopf lifts of 
$I_{1}^{'}$-holomorphic curves in $\mathbb{C}P^{3}$, 
where $I_{1}^{'}$ is an almost complex structure on $\mathbb{C}P^{3}$ given by (\ref{almcpxstr CP3}).
In particular, 
$A_{2}$ (resp. $A_{3})$ is a Hopf lift of 
a horizontal holomorphic curve 
(resp. a null-torsion $I_{1}^{'}$-holomorphic curve 
defined in Definition \ref{null-torsion def}) in $\mathbb{C}P^{3}$. 
Thus, unfortunately, we cannot find homogeneous examples 
which do not arise from other geometries as in the standard $S^{7}$ case (\cite{Lotay3}). 
It is a further problem 
to find an associative submanifold 
which is not congruent to 
the fiber of $S^{7} \rightarrow S^{4}$ or 
the Hopf lift of an $I_{1}^{'}$-holomorphic curve in $\mathbb{C}P^{3}$ 
by the ${\rm Sp}(1){\rm Sp}(2)$-action. 
As far as the author is aware, 
such examples 
are not known so far. 
\end{rem}

However, 
by virtue of this property, 
we can explain their associative deformations. 
\begin{thm} \label{deform summary}
The associative deformations of $L_{1}, L_{2}$, and $A_{1}$ are trivial, i.e. 
all the associative deformations come from the ${\rm Sp}(1) {\rm Sp}(2)$-action, 
while $A_{2}$ and $A_{3}$ have nontrivial associative deformations. 

All the associative deformations of $A_{2}$ consist of 
deformations of $p_{1} (A_{2})$ as a horizontal holomorphic curve, 
i.e. 
those from the ${\rm PGL}(4,\mathbb{C})$-action on $\mathbb{C}P^{3}$ via the Hopf lift, 
and those from actions of $j, k \in {\rm Sp}(1)$, 
where $p_{1}: S^{7} \rightarrow \mathbb{C}P^{3}$ is a projection. 

All the associative deformations of $A_{3}$ consist of 
deformations of $p_{1} (A_{3})$ as a null-torsion holomorphic curve, 
and 
those from actions of 
$j, k \in {\rm Sp}(1)$.
\end{thm}

\begin{rem}
The deformations of the associative submanifolds in the standard $S^{7}$
are studied by the author (\cite{K deform}). 
We could not explain the deformation space of 
the associative submanifold corresponding to $A_{3}$, 
which did not arise from other known geometries. 
However, 
in the squashed $S^{7}$ case, 
the associative deformations of $A_{3}$ are explained 
by the property in Remark \ref{property homog asso}. 
We use the one-to-one correspondence 
between null-torsion $I_{1}^{'}$-holomorphic curves 
and horizontal holomorphic curves in $\mathbb{C}P^{3}$ (\cite{Xu pseudo}). 
\end{rem}

This paper is organized as follows. 
In Section 2, we review the fundamental
facts of $G_{2}$ and ${\rm Spin}(7)$ geometry. 
In Section 3, 
we review the canonical variation and summarize some useful equations. 
In Section 4, we apply it to the 7-sphere $S^{7}$ 
and describe the nearly parallel $G_{2}$-structure
on the squashed $S^{7}$ explicitly. 
Then we give 
basic examples of associative submanifolds in the squashed $S^{7}$. 
In Section 5, 
we prove Theorem \ref{plane asso} by choosing a ``good" frame by 
${\rm Sp}(1) {\rm Sp}(2)$-action. 
In Section 6, 
we prove Theorem \ref{classification homog asso} 
as an analogue of \cite{Lotay3}, \cite {Mashimo}. 
In Section 7, 
we prove Theorem \ref{deform summary} 
by using the representation theory as \cite{K deform}, \cite{Ohnita_def}. 

\noindent{{\bf Acknowledgements}}: 
The author would like to thank Professor Katsuya Mashimo 
for his valuable advice about the representation theory.


\section{Preliminaries}

\subsection{$G_{2}$ and ${\rm Spin}(7)$ geometry}

\begin{definition} \label{def on R7}
Define a $3$-form $\varphi_{0}$ on $\mathbb{R}^{7}$ by
\begin{eqnarray*}
\varphi_{0} = dx_{123} +dx_{1} (dx_{45} +dx_{67}) +dx_{2}(dx_{46} - dx_{57}) - dx_{3}(dx_{47} + dx_{56}), 
\end{eqnarray*}
where  $(x_{1}, \cdots, x_{7})$ is the standard coordinate 
on $\mathbb{R}^{7}$ 
and wedge signs are omitted. 
The Hodge dual of $\varphi_{0}$ is given by 
\begin{eqnarray*}
*\varphi_{0} = dx_{4567} +dx_{23} (dx_{67} + dx_{45}) +dx_{13}(dx_{57} - dx_{46}) - dx_{12}(dx_{56} + dx_{47}). 
\end{eqnarray*}

Decompose $\mathbb{R}^{8} = \mathbb{R} \oplus \mathbb{R}^{7}$
and denote by $x_{0}$ the coordinate on $\mathbb{R}$. 
Define a self-dual $4$-form $\Phi_{0}$ on $\mathbb{R}^{8}$ by
\begin{align*}
\Phi_{0} = dx_{0} \wedge \varphi_{0} + * \varphi_{0}. 
\end{align*}
If we identify $\mathbb{R}^{8} \cong \mathbb{C}^{4}$ via 
$\mathbb{R}^{8} \ni (x_{0}, \cdots, x_{7}) 
\mapsto 
(x_{0} + i x_{1}, x_{2} + i x_{3}, x_{4} + i x_{5}, x_{6} + i x_{7}) 
=: (z_{1}, z_{2}, z_{3}, z_{4}) \in \mathbb{C}^{4}
$, 
then $\Phi_{0}$ is described as 
\begin{align*}
\Phi_{0} = \frac{1}{2} \omega_{0} \wedge \omega_{0} + {\rm Re} \Omega_{0}, 
\end{align*}
where $\omega_{0} = \frac{i}{2} \sum_{j = 1}^{4} dz_{j \overline{j}}$ and 
$\Omega_{0} = dz_{1234}$ are the standard 
K\"{a}hler form and the holomorphic volume form on $\mathbb{C}^{4}$, respectively. 
\end{definition}

The stabilizers of $\varphi_{0}$ and $\Phi_{0}$ are 
the exceptional Lie group $G_{2}$ and ${\rm Spin}(7)$, respectively: 
\begin{eqnarray*}
G_{2} = \{ g \in GL(7, \mathbb{R})  ;  g^{*}\varphi_{0} = \varphi_{0} \}, \qquad 
{\rm Spin}(7) = \{ g \in GL(8, \mathbb{R})  ;  g^{*}\Phi_{0} = \Phi_{0} \}. 
\end{eqnarray*}

The Lie group $G_{2}$ fixes 
the standard metric $g_{0} = \sum^7_{i=1} (dx_{i})^{2}$ and the orientation on $\mathbb{R}^{7}$. 
They are uniquely determined by $\varphi_{0}$ via 
\begin{eqnarray}\label{g varphi}
6 g_{0}(v_{1}, v_{2}) {\rm vol}_{g_{0}} = i(v_{1})\varphi_{0} \wedge i(v_{2})\varphi_{0} \wedge \varphi_{0}, 
\end{eqnarray}
where ${\rm vol}_{g_{0}}$ is a volume form of $g_{0}$, 
$i( \cdot )$ is the interior product, and $v_{i} \in T(\mathbb{R}^{7})$.

Similarly, 
${\rm Spin}(7)$ fixes 
the standard metric $h_{0} = \sum^7_{i=0} (dx_{i})^{2}$ and the orientation on $\mathbb{R}^{8}$. 
They are uniquely determined by $\Phi_{0}$ via 
\begin{eqnarray} \label{g Phi}
\Phi_{0} ^{2} = 14 {\rm vol}_{h_{0}}, \qquad 
(i(w_{2}) i(w_{1}) \Phi_{0})^{2} \wedge \Phi_{0} = 6 
\|  w_{1} \wedge w_{2} \|_{h_{0}}^{2} 
{\rm vol}_{h_{0}}, 
\end{eqnarray}
where ${\rm vol}_{h_{0}}$ is a volume form of $h_{0}$, 
and $w_{i} \in T(\mathbb{R}^{8})$.

\begin{definition}
Let $Y$ be an oriented 7-manifold and 
$\varphi$ a 3-form on $Y$. 
A 3-form $\varphi$ is called a {\bf $G_{2}$-structure} on $Y$ if 
for each $y \in Y$, there exists an oriented isomorphism between $T_{y}Y$ and $\mathbb{R}^{7}$ 
identifying $\varphi_{y}$ with $\varphi_{0}$. 
From (\ref{g varphi}), $\varphi$ induces the metric $g$ 
and the volume form on $Y$.
A  $G_{2}$-structure $\varphi$ is said to be {\bf nearly parallel} 
if $d \varphi = 4 *  \varphi$. 
We call a manifold with a nearly parallel $G_{2}$-structure 
a {\bf nearly parallel $G_{2}$-manifold} for short. 
A  $G_{2}$-structure $\varphi$ is called {\bf torsion-free} if 
$ d\varphi = 0, d*\varphi = 0$.

Let $X$ be an oriented 8-manifold and 
$\Phi$ a 4-form on $X$. 
A 4-form $\Phi$ is called a {\bf ${\rm Spin}(7)$-structure} on $X$ if 
for each $x \in X$, there exists an oriented isomorphism between $T_{x}X$ and $\mathbb{R}^{8}$ 
identifying $\Phi_{x}$ with $\Phi_{0}$. 
From (\ref{g Phi}), $\Phi$ induces the metric $h$ and 
the volume form on $X$. 
A ${\rm Spin}(7)$-structure $\Phi$ is called {\bf torsion-free} if 
$d\Phi = 0$. 
\end{definition}

\begin{lem} \cite{Salamon}
A $G_{2}$-structure $\varphi$ is torsion-free if and 
only if {\rm Hol($g$)} $\subset G_{2}$.  
A ${\rm Spin}(7)$-structure $\Phi$ is torsion-free if and 
only if {\rm Hol($h$)} $\subset {\rm Spin}(7)$.  
\end{lem}

\begin{lem} \label{cha_NP}
The 3-form $\varphi$ is a nearly parallel $G_{2}$-structure 
if and only if its Riemannian cone 
$C(Y) = \mathbb{R}_{>0} \times Y$ admits a torsion-free 
${\rm Spin}(7)$-structure 
$\Phi =  r^{3} dr \wedge \varphi + r^{4} * \varphi$ 
with the induced cone metric 
$\overline{g} = dr^{2} + r^{2} g$. 
\end{lem}

Next, we give a summary of the facts 
about the submanifolds. 
Let $Y$ be a manifold with a $G_{2}$-structure $\varphi$ and the induced metric $g$.

\begin{lem} \cite{Harvey Lawson} \label{def_asso_coasso}
For every oriented $k$-dimensional subspace $V^{k} \subset T_{p}Y$ 
$(\forall p \in Y, k = 3, 4),$ 
we have
$
\varphi|_{V^{3}} \leq {\rm vol}_{V^{3}}, \ 
*\varphi|_{V^{4}} \leq {\rm vol}_{V^{4}}.
$
An oriented 3-submanifold $L^{3} \subset Y$ is called {\bf associative} 
if $\varphi|_{TL^{3}} = {\rm vol}_{L^{3}}$.
An oriented 4-submanifold $L^{4}$ is called {\bf coassociative} 
if $*\varphi|_{TL^{4}} = {\rm vol}_{L^{4}}$.
\end{lem}

\begin{lem} \cite{Harvey Lawson}  \label{equiv condi asso coasso}
An oriented 3-submanifold $L^{3}$ is associative if and only if 
$* \varphi (v_{1}, v_{2}, v_{3}, \cdot) = 0$ for any $v_{j} \in TL^{3}$. 
An oriented 4-submanifold $L^{4}$ is coassociative if and only if 
$\varphi|_{TL^{4}} = 0$. 
\end{lem}

\begin{rem} \label{def of cross product}
Define the cross product $\times : TY \times TY \rightarrow TY$ by 
\begin{align*}
g(u \times v, w) = \varphi (u, v, w)
\end{align*}
for $u, v, w \in TY$. 
When $L^{3}$ is associative, 
there exists an orthonormal basis $\{ e_{1}, e_{2}, e_{3} \}$ 
satisfying $e_{3} = e_{1} \times e_{2}$ 
at any point in $L^{3}$.
\end{rem}

\begin{definition}
Let $X$ be a manifold with a ${\rm Spin}(7)$-structure $\Phi$.  
Then for 
every oriented $4$-dimensional subspace $W \subset T_{x}X$ 
$(\forall x \in X),$ 
we have
$
\Phi|_{W} \leq {\rm vol}_{W}. 
$
An oriented 4-submanifold $N \subset X$ is called {\bf Cayley} 
if $\Phi|_{TN} = {\rm vol}_{N}$.
\end{definition}

\begin{lem}
Let $(Y, \varphi, g)$ be a nearly parallel $G_{2}$-manifold 
and $L \subset Y$ be an oriented 3-submanifold. 
By Lemma \ref{cha_NP}, 
$C(Y)$ is a manifold with a torsion-free ${\rm Spin}(7)$-structure $\Phi$. 
Then $L \subset Y$ is associative if and only if 
$C(L) \subset C(Y)$ is Cayley. 
\end{lem}

\begin{lem} \cite{Lotay3}
There are no coassociative submanifolds of a nearly parallel $G_{2}$-manifold $(Y, \varphi, g)$. 
\end{lem}
\begin{proof}
If $L$ is a coassociative submanifold, we have $\varphi|_{TL} = 0$, 
which implies that $4 {\rm vol}_{L} = 4 * \varphi|_{TL} = d \varphi|_{TL} = 0$. This is a contradiction. 
\end{proof}


\section{Canonical variation}

\subsection{Riemannian submersion}

We give a summary of Chapter 9 of \cite{Besse}.
Let $(M, g)$ and $(B, h)$ be Riemannian manifolds and 
suppose that there exists a Riemannian submersion
$\pi: (M, g) \rightarrow (B, h)$. 
Decompose the tangent bundle 
$TM = \mathcal{V} \oplus \mathcal{H}$, 
where 
a vertical distribution $\mathcal{V}$ is a vector subbundle 
tangent to the fibers $\pi: M \rightarrow B$, 
and a horizontal distribution $\mathcal{H}$ is 
the orthogonal complement bundle of $\mathcal{V}$. 
Denote by $\nabla$  
the Levi-Civita connection 
of $g$.

\begin{definition} \label{tensors A T}
Define (1,2)-tensors $A, T \in C^{\infty}(M, \otimes^{2} T^{*}M \otimes TM)$ by 
\begin{align*}
A_{E} F = (\nabla_{E^{\top}} F^{\perp})^{\top} + (\nabla_{E^{\top}} F^{\top})^{\perp}, \qquad
T_{E} F = (\nabla_{E^{\perp}} F^{\perp})^{\top} + (\nabla_{E^{\perp}} F^{\top})^{\perp}, 
\end{align*}
for $E, F \in \mathfrak{X} (M)$, 
where $\top : TM \rightarrow \mathcal{H}$ and $\perp : TM \rightarrow \mathcal{V}$ 
are projections. 
\end{definition}

\begin{rem}
The distribution $\mathcal{H}$ is involutive if and only if $A \equiv 0$. 
The fibers of $\pi: M \rightarrow B$ are totally geodesic if and only if $T \equiv 0$. 
\end{rem}

In the following, we suppose that \underline{$T \equiv 0$}.

\begin{lem}
Let $X, Y$ be the horizontal vector fields, 
$U, V$ be the vertical vector fields, 
and 
$E, F$ be any vector fields on $M$. 
We have 
\begin{align*}
&A_{U} X =0,  
&A_{U} V &= 0, 
&&A_{X} U = (\nabla_{X} U)^{\top}, \qquad
A_{X} Y = (\nabla_{X} Y)^{\perp}, \\
&A_{X} Y = -A_{Y} X, 
&A_{X} Y &= \frac{1}{2} [X, Y]^{\perp}, 
&&g(A_{X} E, F) = - g(E, A_{X} F). 
\end{align*}
which implies that 
\begin{align*}
&\nabla_{U} V = (\nabla_{U} V)^{\perp}, 
&\nabla_{U} X &= (\nabla_{U} X)^{\top}, \\
&\nabla_{X} U = (\nabla_{X} U)^{\perp} + A_{X} U,
&\nabla_{X} Y &= A_{X} Y + (\nabla_{X} Y)^{\top}. 
\end{align*}
\end{lem}

\subsection{Canonical Variation}

For $s, t > 0$, define the {\bf canonical variation} $\tilde{g}$ of the Riemannian 
metric $g$ on $M$ by 
\begin{align*}
\tilde{g}|_{\mathcal{V} \times \mathcal{V}} = s^{2} g|_{\mathcal{V} \times \mathcal{V}}, \qquad 
\tilde{g}|_{\mathcal{H} \times \mathcal{H}} = t^{2} g|_{\mathcal{H} \times \mathcal{H}}, \qquad
\tilde{g}|_{\mathcal{H} \times \mathcal{V}} = 0. 
\end{align*}

\begin{rem}
Usually, we set $t=1$ for simplicity. 
However, we introduce a parameter $t$ 
to define the nearly parallel $G_{2}$-structure. 
See Proposition \ref{can var}.
\end{rem}

Denote by $\tilde{\nabla}$ the Levi-Civita connection of $\tilde{g}$. 
Set (1,2)-tensors $\tilde{A}$ and $\tilde{T}$ as in Definition \ref{tensors A T}. 

\begin{rem}
The assumption $T \equiv 0$ implies that $\tilde{T} \equiv 0$ for all $s, t > 0$. 
\end{rem}

Under the canonical variation, 
the tensor $A$ in Definition \ref{tensors A T}  
and the Levi-Civita connection are changed as follows. 
\begin{lem} \label{relation of A and tildeA}
Let $X, Y$ be the horizontal vector fields, 
and $U, V$ be the vertical vector fields on $M$. 
We have 
\begin{align*}
&\tilde{A}_{X} Y = A_{X} Y,  & \tilde{A}_{X} U &= \frac{s^{2}}{t^{2}} A_{X} U, \\
&\tilde{\nabla}_{X} Y = \nabla_{X} Y, 
&\tilde{\nabla}_{U} V &= \nabla_{U} V, 
\end{align*}
\begin{align*}
\tilde{\nabla}_{X} U &= \frac{s^{2}}{t^{2}} (\nabla_{X} U)^{\top} + (\nabla_{X} U)^{\perp},\\
\tilde{\nabla}_{U} X &= \frac{s^{2}}{t^{2}} (\nabla_{U} X)^{\top} 
+ \left( 1 - \frac{s^{2}}{t^{2}} \right) [U, X]^{\top}.
\end{align*}
\end{lem}

This lemma implies the following useful equation.

\begin{lem} \label{relation of connection}
For $E_{1}, E_{2} \in \mathfrak{X}(M)$, we have 
\begin{align*} 
\tilde{\nabla}_{E_{1}} E_{2} - \nabla_{E_{1}} E_{2} = 
\left( -1+ \frac{s^{2}}{t^{2}} \right) (A_{E_{1}} E_{2}^{\perp} + A_{E_{2}} E_{1}^{\perp}).
\end{align*}
\end{lem}

\section{Nearly parallel $G_{2}$-structure on the squashed $S^{7}$}

The standard $S^{7}$ admits a canonical nearly parallel $G_{2}$-structure. 
By the canonical variation, we obtain the second nearly parallel $G_{2}$-structure 
on $S^{7}$ (Proposition \ref{can var}). 
First, we review a 3-Sasakian structure on $S^{7}$. 

\subsection{3-Sasakian structure on $S^{7}$}

Consider the following Lie groups: 
\begin{align*}
{\rm Sp}(1) &= \{ a_{1} + a_{2}j \in \mathbb{H}; a_{i} \in \mathbb{C}, |a_{1}|^{2} +  |a_{2}|^{2} = 1 \}, \\ 
{\rm Sp}(2) &= 
\{ g \in {\rm GL}(2, \mathbb{H}); g \mbox{ preserves the metric on } \mathbb{H}^{2}  \}\\ 
&=\{ g \in {\rm U}(4); {}^t\! gJg=J \} \\
&=  \{ \left(u,  J \overline{u}, v, J\overline{v} \right) ; 
u,v \in \mathbb{C}^{4}, |u| = |v| = 1, 
\langle v, u \rangle_{\mathbb{C}} = \langle v, J \overline{u} \rangle_{\mathbb{C}} = 0
\}, 
\end{align*}
where $J = 
\left(
\begin{array}{cc}
J' & 0 \\
0 & J' \\
\end{array} 
\right), 
J'= 
\left(
\begin{array}{cc}
0 & -1 \\
1 & 0 \\
\end{array} 
\right), 
$
and $\langle \cdot, \cdot \rangle_{\mathbb{C}}: \mathbb{C}^{4} \times \mathbb{C}^{4} \rightarrow \mathbb{C}$ 
is the standard Hermitian metric on $\mathbb{C}^{4}$.

Let ${\rm Sp}(1) \times {\rm Sp}(2)$ act on  $\mathbb{H}^{2}$ by 
\begin{align*} 
(q, A) \cdot (q_{1}, q_{2}) = q (q_{1}, q_{2})  {}^t\! \overline{A}, 
\end{align*}
where $(q, A) \in {\rm Sp}(1) \times {\rm Sp}(2), (q_{1}, q_{2}) \in \mathbb{H}^{2}$. 
Via the identification 
$\mathbb{C}^{4} \ni (z_{1}, \cdots, z_{4}) \mapsto (z_{1} + z_{2}j, z_{3} + z_{4}j) \in \mathbb{H}^{2}$, 
the ${\rm Sp}(1)$-action on $\mathbb{C}^{4}$ is described as 
\begin{align} \label{right sp1 action}
(a_{1} + a_{2}j) \cdot u = a_{1} u + a_{2} J \overline{u}, 
\end{align}
where $u \in \mathbb{C}^{4}$, 
and ${\rm Sp}(2) \subset {\rm U}(4)$ acts on $\mathbb{C}^{4}$ canonically. 
By definition, the ${\rm Sp}(1)$-action commutes with the ${\rm Sp}(2)$-action. 

The actions of $i, j, k \in {\rm Sp}(1)$ induce 
complex structures $I_{1}, I_{2}, I_{3}$ on $\mathbb{C}^{4}$, respectively, 
and hence induce the 3-Sasakian structure $\{ (\Phi_{i}, \xi_{i}, \eta_{i}, g) \}_{i =1,2,3}$ on $S^{7}$, 
where $g$ is the standard metric on $S^{7}$, 
and 
a vector field $\xi_{i} \in \mathfrak{X}(S^{7})$, 
a 1-form $\eta_{i} \in \Omega^{1}(S^{7})$, 
and a (1, 1)-tensor $\Phi_{i} \in C^{\infty}(S^{7}, {\rm End}(TS^{7}))$ are defined by 
\begin{align*}
(\xi_{i})_{z} &= - I_{i} (z), \mbox{ where } z \in \mathbb{C}^{4}, \\
\eta_{i} &= g(\xi_{i}, \cdot), \\
\Phi_{i} &= 
\left\{ \begin{array}{ll}
I_{i} & (\mbox{on } {\rm Ker} \eta_{i}) \\
0   & (\mbox{on } \mathbb{R} \xi_{i}). \\
\end{array} \right.
\end{align*}
Note that the following conditions are satisfied: 
\begin{align*}
\Phi_{i+2} &= \Phi_{i} \circ \Phi_{i+1} - \eta_{i+1} \otimes \xi_{i} 
= - \Phi_{i+1} \circ \Phi_{i} + \eta_{i} \otimes \xi_{i+1}, \\
\xi_{i+2}   &= \Phi_{i} (\xi_{i+1}) = - \Phi_{i+1} (\xi_{i}), \\
\eta_{i+2} &= \eta_{i} \circ \Phi_{i+1} = -\eta_{i+1} \circ \Phi_{i}, 
\end{align*}
where $i \in \mathbb{Z}/3$.
These tensors are described explicitly as follows.

\begin{lem} \label{xi eta S7}
\begin{align*}
\xi_{1} &= -i {}^t\!(z_{1}, z_{2}, z_{3}, z_{4}), \\
\xi_{2} &=  {}^t\!(\overline{z}_{2}, - \overline{z}_{1}, \overline{z}_{4}, -\overline{z}_{3}), \\ 
\xi_{3} &=  i {}^t\!(\overline{z}_{2}, - \overline{z}_{1}, \overline{z}_{4}, -\overline{z}_{3}), 
\end{align*}
\begin{align*}
\eta_{1} &= {\rm Im}\left( \sum_{j = 1}^{4} z_{j} d \overline{z}_{j} \right), \qquad
\eta_{2} + i \eta_{3} = -z_{1} d z_{2} + z_{2} dz_{1} -z_{3} d z_{4} + z_{4} dz_{3}, \\ 
d \eta_{1} &= -i \sum_{j = 1}^{4} dz_{j \overline{j}} = -2 g(\Phi_{1} (\cdot), \cdot), \qquad
d (\eta_{2} + i \eta_{3}) = -2 (dz_{12} + dz_{34}). 
\end{align*}
\end{lem}

\subsection{Second nearly parallel $G_{2}$-structure on $S^{7}$}

Applying the canonical variation to a Riemannian submersion $\pi : S^{7} \rightarrow S^{4} = \mathbb{H}P^{1}$, 
we obtain the second nearly parallel $G_{2}$-structure 
$(\tilde{\varphi}, \tilde{g})$ on $S^{7}$. 
Denote by $\omega_{i} = \frac{1}{2} D \eta_{i} 
= \frac{1}{2} d \eta_{i} ((\cdot)^{\top}, (\cdot)^{\top}) \in \Omega^{2}(S^{7})$ 
the covariant differentiation of $\frac{1}{2} \eta_{i}$,
where $\top: TS^{7} \rightarrow \mathcal{H}$ is a canonical projection.  
In other words, we have 
\begin{align*}
\omega_{1} = \frac{1}{2} d \eta_{1} + \eta_{23}, \qquad 
\omega_{2} = \frac{1}{2} d \eta_{2} + \eta_{31}, \qquad 
\omega_{3} = \frac{1}{2} d \eta_{3} + \eta_{12}.   
\end{align*}
since $[\xi_{i}, \xi_{i+1}] = 2 \xi_{i+2}$ for $i \in \mathbb{Z}/3.$
On the other hand, 
it is well-known that $\frac{1}{2} d \eta_{i} = - g(\Phi_{i}(\cdot), \cdot)$. 
For example,  see Section 2 of \cite{Ohnita_def}.
Then we deduce that 
\begin{align*}
\omega_{i}= -g(\Phi_{i}(\cdot)^{\top}, (\cdot)^{\top}) \qquad
\mbox{ for } i = 1,2,3.
\end{align*}

\begin{rem} \label{omega ptwise}
Take any unit vector $X_{0} \in \mathcal{H}$ and set $X_{i} = \Phi_{i}(X_{0})$ for $i = 1, 2, 3$.
Denote by $\{ X^{i} \}$ the dual of $\{ X_{i} \}$. Then we have 
\begin{align*}
\omega_{1} = -(X^{01} + X^{23}), \qquad
\omega_{2} = -(X^{02} + X^{31}), \qquad
\omega_{3} = -(X^{03} + X^{12}). 
\end{align*}
\end{rem}

\begin{prop} \cite{FKMS} \label{can var}
Define the Riemannian metric $\tilde{g}$, 
a 3-form $\tilde{\varphi} \in \Omega^{3}(S^{7})$, and the 4-form $* \tilde{\varphi} \in \Omega^{4}(S^{7})$
on $S^{7}$ by 
\begin{align*}
\tilde{g}|_{\mathcal{V} \times \mathcal{V}} 
= \left( \frac{3}{5} \right)^{2} g|_{\mathcal{V} \times \mathcal{V}}, \qquad
\tilde{g}|_{\mathcal{H} \times \mathcal{H}} = \left( \frac{3}{\sqrt{5}} \right)^{2} g|_{\mathcal{H} \times \mathcal{H}}, \qquad
\tilde{g}|_{\mathcal{H} \times \mathcal{V}} = 0, 
\end{align*}
\begin{align*}
\tilde{\varphi} &= \frac{27}{25} \left( \frac{1}{5} \eta_{123} + \sum_{i = 1}^{3} \eta_{i} \wedge \omega_{i} \right), \\
* \tilde{\varphi} &= \frac{27}{25} \left( \frac{1}{2} \sum_{i = 1}^{3} \omega_{i}^{2} 
+ \frac{3}{5} \left( \eta_{23} \wedge \omega_{1} + \eta_{31} \wedge \omega_{2} 
+ \eta_{12} \wedge \omega_{3} \right) \right). 
\end{align*}
Then $(\tilde{\varphi}, \tilde{g})$ is a nearly parallel $G_{2}$-structure with 
${\rm Hol}(\overline{\tilde{g}}) = {\rm Spin}(7)$ 
and $* \tilde{\varphi}$ is a Hodge dual of $\tilde{\varphi}$ with respect to $\tilde{g}$. 
We call $(S^{7}, \tilde{\varphi}, \tilde{g})$ the {\bf squashed $S^{7}$}. 
\end{prop}

\begin{proof}[Outline of the proof]
Set 
\begin{align*}
\tilde{g}|_{\mathcal{V} \times \mathcal{V}} = s^{2} g|_{\mathcal{V} \times \mathcal{V}}, \qquad
\tilde{g}|_{\mathcal{H} \times \mathcal{H}} = t^{2} g|_{\mathcal{H} \times \mathcal{H}}, \qquad
\tilde{g}|_{\mathcal{H} \times \mathcal{V}} = 0, 
\end{align*}
\begin{align*}
\tilde{\varphi} &= s^{3} \eta_{123} + s t^{2} \sum_{i = 1}^{3} \eta_{i} \wedge \omega_{i}, 
\end{align*}
for $s, t > 0$. We find $s, t >0$ satisfying $d \tilde{\varphi} = 4 * \tilde{\varphi}$. 
Setting  
$G_{1} = s^{3} \eta_{123}, G_{2} = s t^{2} \sum_{i = 1}^{3} \eta_{i} \wedge \omega_{i}$, we have  
\begin{align*}
* G_{1} = \frac{t^{4}}{6} \sum_{i = 1}^{3} \omega_{i}^{2}, \qquad
* G_{2} = s^{2} t^{2} (\eta_{23} \wedge \omega_{1} + \eta_{31} \wedge \omega_{2} 
+ \eta_{12} \wedge \omega_{3}), \\ 
d(\eta_{123}) = \frac{2}{s^{2} t^{2}} *G_{2}, \qquad
d \left( \sum_{i = 1}^{3} \eta_{i} \wedge \omega_{i} \right) = 
\frac{12}{t^{4}} *G_{1} + \frac{2}{s^{2} t^{2}} * G_{2}. 
\end{align*}
Then we see that 
$d \tilde{\varphi} = \frac{12}{s} *G_{1} + (\frac{2s}{t^{2}} + \frac{2}{s}) * G_{2}$, 
and hence $d \tilde{\varphi} = 4* \tilde{\varphi}$ is equivalent to 
$s = 3/5, t = 3/ \sqrt{5}$. 
The metric $\tilde{g}$ is not Sasaki-Einstein, 
and hence satisfies ${\rm Hol}(\overline{\tilde{g}}) = {\rm Spin}(7)$ 
by the classification of the dimensions of the spaces of real Killing spinors. 
\end{proof}

\begin{rem} \cite{FKMS}
Proposition \ref{can var} is valid for any compact 3-Sasakian manifolds. 
The metric $\tilde{g}$ is Einstein if and only if 
$s=t$ or $s= t/ \sqrt{5}$. 
\end{rem}

Since 
$
\eta_{1} = {\rm Im}({}^t\! z d \overline{z}), 
\eta_{2} + i \eta_{3} = -d {}^t\! z \cdot Jz,
$
where $z = {}^t\! (z_{1}, z_{2}, z_{3}, z_{4})$, 
${\rm Sp}(2)$ preserves $\eta_{j} (j=1,2,3)$. 
For $q=a_{1}+a_{2}j \in {\rm Sp}(1)$, 
we have 
$
(q^{*} \eta_{1}, q^{*} \eta_{2}, q^{*} \eta_{3}) = (\eta_{1}, \eta_{2}, \eta_{3}) {}^t\! M_{q},
$
where $M_{q} \in {\rm SO}(3)$ is described as 
\begin{align*}
M_{q}
=
\left(
\begin{array}{ccc}
|a_{1}|^{2} - |a_{2}|^{2} & 2 {\rm Im}(a_{1} \overline{a}_{2}) & 2 {\rm Re} (a_{1} \overline{a}_{2})\\
2 {\rm Im}(a_{1} a_{2})& {\rm Re}(a_{1}^{2} + a_{2}^{2})     & {\rm Im}(-a_{1}^{2} + a_{2}^{2}) \\
-2 {\rm Re}(a_{1} a_{2})&{\rm Im}(a_{1}^{2} + a_{2}^{2})     &{\rm Re}(a_{1}^{2} - a_{2}^{2}) \\
\end{array} 
\right).
\end{align*}
Hence we see that 
${\rm Sp}(2)$ and ${\rm Sp}(1)$ preserve 
$g |_{\mathcal{H} \times \mathcal{H}}, g |_{\mathcal{V} \times \mathcal{V}}, \tilde{g}$ and $\tilde{\varphi}$. 
In fact, we have the following. 

\begin{lem} \cite{FKMS} \label{auto sq S7}
The automorphism group of the squashed $(S^{7}, \tilde{\varphi}, \tilde{g})$ is 
${\rm Sp}(1) {\rm Sp}(2) = {\rm Sp}(1) \times {\rm Sp}(2)/ \{ \pm (1, 1) \}$. 
\end{lem}

\begin{rem}
In this paper, we often consider the subgroup of ${\rm Sp}(1) {\rm Sp}(2)$. 
If there may be some confusion, 
denoting ${\rm Sp}(1) = {\rm Sp}(1)_{L}$ and ${\rm Sp}(2) = {\rm Sp}(2)_{R}$, 
we distinguish subgroups of ${\rm Sp}(1) {\rm Sp}(2) = {\rm Sp}(1)_{L} {\rm Sp}(2)_{R}$.
\end{rem}

\begin{lem} \label{relation of connection 3-sasaki}
For any $E_{1}, E_{2} \in \mathfrak{X}(S^{7})$, we have 
\begin{align*}
\tilde{g} (E_{1}, E_{2}) &= 
- \frac{36}{25} \sum_{j=1}^{3} \eta_{j} (E_{1}) \eta_{j} (E_{2})
+ \frac{9}{5} g(E_{1}, E_{2}), \\
\tilde{\nabla}_{E_{1}} E_{2} - \nabla_{E_{1}} E_{2} &= 
\frac{4}{5} \Theta (E_{1}, E_{2}),
\end{align*}
where 
$\Theta \in C^{\infty}(S^{7}, \otimes^{2} T^{*}S^{7} )$ is defined by 
\begin{align*}
\Theta (E_{1}, E_{2}) = 
\sum_{i=1}^{3} \left( \eta_{i} (E_{1}) \Phi_{i} (E_{2}) + \eta_{i} (E_{2}) \Phi_{i} (E_{1}) \right).
\end{align*}
\end{lem}

\begin{proof}
The first equation is proved easily and we omit the proof.
Set $(s, t) = \left( 3/5, 3/ \sqrt{5} \right)$ in Lemma \ref{relation of connection}. 
Since $A_{X} U = - \sum_{i=1}^{3} \eta_{i}(U) \Phi_{i} (X)$ 
for a horizontal vector $X$ and a vertical vector $U$,  
we have 
\begin{align*}
\tilde{\nabla}_{E_{1}} E_{2} - \nabla_{E_{1}} E_{2} = 
\frac{4}{5} 
\sum_{i=1}^{3} \left( \eta_{i} (E_{1}) \Phi_{i} (E_{2}^{\top}) + \eta_{i} (E_{2}) \Phi_{i} (E_{1}^{\top}) \right).
\end{align*}
We easily see that 
the right hand side is equal to $\frac{4}{5} \Theta (E_{1}, E_{2})$. 
\end{proof}


\subsection{Associative submanifolds of the squashed $S^{7}$}

By the definition of $\tilde{\varphi}$ in Proposition \ref{can var}, we see the following. 
\begin{rem}
There are no horizontal associative submanifolds, 
i.e. 
there are no associative submanifolds whose 
tangent spaces are contained in $\mathcal{H}$. 
\end{rem}

Let 
$\pi : S^{7} \rightarrow S^{4}$ and 
$p_{1}: S^{7} \rightarrow \mathbb{C}P^{3}$ 
be the Hopf fibrations and 
$p_{2}: \mathbb{C}P^{3} \rightarrow S^{4}$ 
be the twistor fibration satisfying $\pi = p_{2} \circ p_{1}$. 
Denote by $\underline{\mathcal{V}}$ and $\underline{\mathcal{H}}$ 
the distributions of $\mathbb{C}P^{3}$ induced by $\mathcal{V}$ and $\mathcal{H}$, respectively. 
In other words, $\underline{\mathcal{V}}$ is a vector subbundle of $T \mathbb{C} P^{3}$ 
tangent to the fibers $p_{2}$, and $\underline{\mathcal{H}}$ is 
the orthogonal complement bundle of $\underline{\mathcal{V}}$. 
By an abuse of notation,
denote by $I_{1}$ the standard complex structure 
on $\mathbb{C}P^{3}$ induced from 
the standard complex structure $I_{1}$ on $\mathbb{C}^{4}$. 
Define the almost complex structure $I_{1}^{'}$ on $\mathbb{C}P^{3}$ by 
\begin{align} \label{almcpxstr CP3}
I_{1}^{'}|_{\underline{\mathcal{V}}} = -I_{1}|_{\underline{\mathcal{V}}}, \qquad
I_{1}^{'}|_{\underline{\mathcal{H}}} = I_{1}|_{\underline{\mathcal{H}}}. 
\end{align}
The almost complex structure $I_{1}^{'}$ is never integrable, and defines 
the nearly K\"{a}hler structure on $\mathbb{C}P^{3}$.

\begin{prop} \label{example asso psehol}
Let $\Sigma \subset \mathbb{C}P^{3}$ be an $I_{1}^{'}$-holomorphic curve. 
Then the Hopf lift $p_{1}^{-1}(\Sigma) \subset S^{7}$ of $\Sigma$ is associative in the squashed $S^{7}$. 
\end{prop}

\begin{proof}
Use the notation of Remark \ref{omega ptwise} and Proposition \ref{can var}. 
Setting $\tilde{\eta}_{i} = (3/5) \eta_{i}$ and $\tilde{X}^{i} = (3/\sqrt{5}) X^{i}$, 
we have 
\begin{align*}
\tilde{\varphi} = \tilde{\eta}_{1} (\tilde{\eta}_{23} - \tilde{X}^{01} - \tilde{X}^{23}) 
- \tilde{\eta}_{2} (\tilde{X}^{02} + \tilde{X}^{31}) 
- \tilde{\eta}_{3} (\tilde{X}^{03} + \tilde{X}^{12}). 
\end{align*}
Then we obtain 
$\tilde{\eta}_{23} - \tilde{X}^{01} - \tilde{X}^{23} = - \tilde{G} (I_{1}^{'} (\cdot), \cdot)$, 
where $\tilde{G} = \tilde{\eta}_{2} \otimes \tilde{\eta}_{2} + \tilde{\eta}_{3} \otimes \tilde{\eta}_{3} 
+ \sum_{j = 0}^{3} \tilde{X}^{j} \otimes \tilde{X}^{j}$,   
which gives the proof.
\end{proof}

\begin{rem} \label{basic concrete example of asso}

Each fiber $F \cong S^{2}$ of $p_{2}$ is an obvious $I_{1}^{'}$-holomorphic curve. 
Then the Hopf lift $p_{1}^{-1} (F) = \pi^{-1}(*)$ of $F$ 
is associative. This is the intersection of 
a quaternionic plane and $S^{7}$. 

If $\Sigma \subset \mathbb{C}P^{3}$ is a horizontal $I_{1}$-holomorphic curve, 
where we call the curve $\Sigma$ horizontal if $T \Sigma \subset \underline{\mathcal{H}}|_{\Sigma}$, 
$\Sigma \subset \mathbb{C}P^{3}$ is also an $I_{1}^{'}$-holomorphic curve. 
Thus the Hopf lift $p_{1}^{-1}(\Sigma)$ is associative. 
Since $I_{1}$ is the standard complex structure, we know many examples of these curves. 
\end{rem}

\section{Classification of Cayley planes}

In this section, we prove Theorem \ref{plane asso}. 
Let $V^{4} \subset \mathbb{R}^{8}$ be a 4-plane. 
We classify the associative submanifolds of the form $V \cap S^{7}$ 
by choosing a ``good" frame of $V$ by the  ${\rm Sp}(1) {\rm Sp}(2)$-action 
to consider the associative condition. 

Suppose that $V$ is spanned by $e_{0}, \cdots, e_{3}$ $(e_{i} \in \mathbb{C}^{4} = \mathbb{R}^{8})$.
Since ${\rm Sp}(1) {\rm Sp}(2)$ acts transitively on $S^{7}$, 
we may assume that 
\begin{align*}
e_{0} = {}^t\! (1, 0, 0, 0). 
\end{align*}
The stabilizer of ${\rm Sp}(1) {\rm Sp}(2)$ at $e_{0}$ is 
diffeomorphic to ${\rm Sp}(1) {\rm Sp}(1)$, which acts on $S^{7}$ as 
$
[(p, q)] \cdot (q_{1}, q_{2}) = (p q_{1} \overline{p}, p q_{2} \overline{q}), 
$
where $[(p, q)] \in {\rm Sp}(1) {\rm Sp}(1)$ and $(q_{1}, q_{2}) \in S^{7} \subset \mathbb{H}^{2} = \mathbb{R}^{8}$. 
Thus we may assume that 
\begin{align*}
e_{1} = {}^t\! (ci, 0, s, 0), 
\end{align*}
for $c, s \geq 0, c^{2}+s^{2} = 1$. 
Since $\{ [(z, z)]; z \in {\rm U}(1) \} \subset {\rm Sp}(1) {\rm Sp}(1)$ fixes $e_{1}$, 
by sweeping out the first entry, 
we may assume that 
\begin{align*}
e_{2} = {}^t\! (0, A_{2}, A_{3} + i B_{3}, A_{4} + i B_{4}), 
\end{align*}
for $A_{j}, B_{k} \in \mathbb{R}, A_{2} \geq 0$.

\begin{lem} \label{asso in e0}
We have 
\begin{align*}
\frac{5}{3} (e_{1} \times e_{2})_{e_{0}} 
= 
{}^t\! (5 B_{3} s i, 5 A_{4} s +(-A_{2} c + 5 B_{4} s)i, -B_{3} c + A_{3} c i, 
(-B_{4} c -A_{2} s) + A_{4} c i). 
\end{align*}
\end{lem}

Thus denoting by $e_{3}$ the left-hand side, 
we see that ${\rm span}_{\mathbb{R}} \{ e_{1}, e_{2}, e_{3}\} \subset T_{e_{0}} S^{7}$ is associative. 
We deduce the condition by calculating $* \tilde{\varphi} (e_{i}, e_{j}, e_{k}, \cdot)_{e_{l}} = 0$ 
in the following cases: 
\begin{enumerate}[{(}1{)}]
\item  $c>0, A_{2} > 0,$
\item  $c>0, A_{2} = 0,$
\item  $c=0.$
\end{enumerate}

\begin{lem} \label{asso in e1}
In the case $(1)$, 
the condition $* \tilde{\varphi} (e_{0}, e_{2}, e_{3}, \cdot)_{e_{1}} = 0$ 
is equivalent to 
\begin{enumerate}[{(}i{)}]
\item
$s=0$, 
\item
$s \neq 0, \quad A_{3}= B_{3} =0, \quad c^{2}-3 s^{2}=0, \quad$ or 
\item
$s \neq 0, \quad A_{3}= B_{3} =0, \quad c(A_{2}^{2}+3 A_{4}^{2}+3 B_{4}^{2}) -2s A_{2} B_{4}=0. $
\end{enumerate}
\end{lem}
 
We abbreviate 
the case that (1) and (ii) hold as the case (1)-(ii) in the following.

\begin{lem} \label{asso in e2}
In the case (1)-(ii) or (1)-(iii), 
by normalizing $e_{2}$, 
we may assume that $A_{2}^{2} + A_{4}^{2} + B_{4}^{2}=1$. 
Then 
$* \tilde{\varphi} (e_{0}, e_{1}, e_{3}, \cdot)_{e_{2}} = 0$ 
is equivalent to 
\begin{enumerate}[{(}a{)}]
\item
$A_{4}=B_{4}=0$, 
\item
$A_{4} =0, \quad A_{2}^{2}-3 B_{4}^{2}=0, \quad$ or 
\item
$A_{4}= 0, \quad (c^{2} + 3 s^{2}) A_{2}-2cs B_{4}=0. $
\end{enumerate}
\end{lem}

\begin{proof}[Proof of Lemma \ref{asso in e0}]
At $e_{0}$, we have 
\begin{align*}
\xi_{1} &= {}^t\! (-i, 0, 0, 0), \\ 
\xi_{2} &= {}^t\! (0, -1, 0, 0), \\ 
\xi_{3} &= {}^t\! (0, -i, 0, 0). 
\end{align*}
Setting $X_{0} = {}^t\! (0, 0, 1, 0), $ we see $X_{0} \in \mathcal{H}_{e_{0}}$.
Then $X_{i} = \Phi_{i} (X_{0})$ for $i=1, 2, 3$ is described as 
\begin{align*}
X_{1} &= {}^t\! (0, 0, i, 0), \\ 
X_{2} &= {}^t\! (0, 0, 0, 1), \\ 
X_{3} &= {}^t\! (0, 0, 0, i), 
\end{align*} 
and we have 
\begin{align*}
e_{1} &= -c \xi_{1} + s X_{0}, \\
e_{2} &= -A_{2} \xi_{2} + A_{3} X_{0} + B_{3} X_{1} + A_{4} X_{2} + B_{4} X_{3}.
\end{align*}
By the definition of $\tilde{\varphi}$ in Proposition \ref{can var}, 
we obtain 
\begin{align*}
\tilde{\varphi} (e_{1}, e_{2}, \cdot)_{e_{0}} = 
\frac{27}{125} c (A_{2} \eta_{3} + 5 A_{3} X^{1} - 5 B_{3} X^{0} + 5A_{4} X^{3} -5 B_{4} X^{2}) \\
+ \frac{27}{25} s (-A_{2} X^{2} - B_{3} \eta_{1} - A_{4} \eta_{2} - B_{4} \eta_{3}). 
\end{align*}
Since $\tilde{g} = \frac{9}{25} \sum_{i = 1}^{3} \eta_{i} + \frac{9}{5} \sum_{a = 0}^{3} X^{a}$, 
we obtain the lemma.
\end{proof}

\begin{proof}[Proof of Lemma \ref{asso in e1}]
As in the proof of Lemma \ref{asso in e0}, we have at $e_{1}$
\begin{align*}
\xi_{1} &= {}^t\! (c, 0, -is, 0), \\ 
\xi_{2} &= {}^t\! (0, ic, 0, -s), \\ 
\xi_{3} &= {}^t\! (0, -c, 0, -is). 
\end{align*}
Setting $X_{0} = {}^t\! (0, is, 0, c) \in \mathcal{H}_{e_{1}}$, 
$X_{i} = \Phi_{i} (X_{0})$ for $i=1, 2, 3$ is described as 
\begin{align*}
X_{1} &= {}^t\! (0, -s, 0, ic), \\ 
X_{2} &= {}^t\! (is, 0, -c, 0), \\ 
X_{3} &= {}^t\! (-s, 0, -ic, 0). 
\end{align*}
Then by a direct computation,  
$* \tilde{\varphi} (e_{0}, e_{2}, e_{3}, \cdot)_{e_{1}} = 0$ is equivalent to 
\begin{align}
4s(c^{2} - 3 s^{2}) (c A_{2}^{2} + 3c A_{4}^{2} + 3c B_{4}^{2} -2s A_{2} B_{4}) = 0, \label{plane e1 1}\\
s 
\left( 
\begin{array}{ccc}
c(-2s^{2} + c^{2}) & -2 s^{3}      & c(3s^{2} + c^{2}) \\
3sc                   & 3s^{2}+c^{2} & -2sc \\
\end{array} 
\right)
\left( 
\begin{array}{c}
A_{3} A_{4} \\
A_{2} B_{3} \\
B_{3} B_{4}
\end{array} 
\right)
= 0, \label{plane e1 2}\\
s 
\left( 
\begin{array}{ccc}
sc & c^{2}-2 s^{2}      & -(3s^{2} + c^{2}) \\
c                   & -3s & -2s \\
\end{array} 
\right)
\left( 
\begin{array}{c}
A_{2} A_{3} \\
A_{3} B_{4} \\
B_{3} A_{4}
\end{array} 
\right)
= 0. \label{plane e1 3}
\end{align}
It is clear that $s=0$ is a solution of (\ref{plane e1 1}), (\ref{plane e1 2}) and (\ref{plane e1 3}). 
We may assume that $s \neq 0$. 
From (\ref{plane e1 2}) and (\ref{plane e1 3}), we have 
\begin{align*}
(A_{3} A_{4}, A_{2} B_{3}, B_{3} B_{4}) &= k (-(c^{2} + 5 s^{2}), 5sc, c^{2}), \\
(A_{2} A_{3}, A_{3} B_{4}, B_{3} A_{4}) &= l (5s, c, c), 
\end{align*}
for $k, l \in \mathbb{R}$. 
Since 
$A_{3} A_{4} B_{3} B_{4} = -k^{2} c^{2} (c^{2} + 5s^{2}) = l^{2} c^{2}$, we obtain $k=l=0$. 
The assumption $A_{2} > 0$ gives $A_{3} = B_{3} = 0$.
\end{proof}

\begin{proof}[Proof of Lemma \ref{asso in e2}]
As in the proof of Lemma \ref{asso in e0}, we have at $e_{1}$
\begin{align*}
\xi_{1} &= {}^t\! (0, -i A_{2}, 0, B_{4} -i A_{4}), \\ 
\xi_{2} &= {}^t\! (A_{2}, 0, A_{4}-i B_{4}, 0), \\ 
\xi_{3} &= {}^t\! (i A_{2}, 0, B_{4} + i A_{4}, 0). 
\end{align*}
Setting $X_{0} = {}^t\! (A_{4} + i B_{4}, 0, -A_{2}, 0) \in \mathcal{H}_{e_{2}}$, 
$X_{i} = \Phi_{i} (X_{0})$ for $i=1, 2, 3$ is described as 
\begin{align*}
X_{1} &= {}^t\! (-B_{4} + i A_{4}, 0, -i A_{2}, 0), \\ 
X_{2} &= {}^t\! (0, A_{4} - i B_{4}, 0, -A_{2}), \\ 
X_{3} &= {}^t\! (0, B_{4} + i A_{4}, 0, -i A_{2}). 
\end{align*}
Then by a direct computation, 
$* \tilde{\varphi} (e_{0}, e_{1}, e_{3}, \cdot)_{e_{2}} = 0$ is equivalent to 
\begin{align}
A_{4} \{ c A_{2} (c A_{2}^{2} -2s A_{2} B_{4} -3 c A_{4}^{2} -3c  B_{4}^{2}) & \nonumber \\ 
+ 
6 B_{4} s (-3 s A_{2} B_{4} + 2c A_{4}^{2} + 2c B_{4}^{2}) \}&= 0, \label{plane e2 1}\\
(c^{2}+ 3s^{2}) A_{2}^{3} B_{4} -2cs A_{2}^{2} B_{4}^{2} + 3 (3s^{2}-c^{2}) A_{2} A_{4}^{2} B_{4}& \nonumber \\
-3 (c^{2}+ 3s^{2}) A_{2} B_{4}^{3} -6cs A_{4}^{4} + 6cs B_{4}^{4} &=0, \label{plane e2 2}\\
s A_{2} A_{4} (c A_{2}^{2} -2s A_{2} B_{4} + 3c A_{4}^{2} + 3c B_{4}^{2}) &= 0. \label{plane e2 3}
\end{align}
Suppose that $A_{4} \neq 0$ for a contradiction. 
Then (\ref{plane e2 3}) implies that 
\begin{align}
c A_{2}^{2} -2s A_{2} B_{4} + 3c A_{4}^{2} + 3c B_{4}^{2} = 0. \label{plane e2 4}
\end{align}
Eliminating $A_{4}^{2}$ and $B_{4}^{2}$ from (\ref{plane e2 1}), we have 
$
2A_{2} (c A_{2} - 5s B_{4}) (c A_{2} + s B_{4}) =0.
$
However, 
the left hand side of (\ref{plane e2 4}) is greater than $0$ when 
$B_{4} = \frac{c}{5 s} A_{2}$ or $- \frac{c}{s} A_{2}$. 
Thus we have $A_{4} = 0$. 

Then the left-hand sides of (\ref{plane e2 1}) and (\ref{plane e2 3}) vanish, 
and that of (\ref{plane e2 2}) is equal to 
$B_{4} (A_{2}^{2} - 3 B_{4}^{2}) \{ (c^{2} + 3 s^{2}) A_{2} -2cs B_{4} \}$, 
hence the proof is done.
\end{proof}

\begin{proof}[Proof of Theorem\ref{plane asso}]
From Lemma \ref{asso in e1} and \ref{asso in e2}, 
we consider the following cases: 

\underline{Case (1)-(i)}
By the ${\rm Sp}(1)$-action, 
we may assume that $B_{3}=A_{4}=B_{4}=0$. 
Normalizing $e_{2}$, we may assume $A_{2}^{2}+A_{3}^{2}=1$. 
Then as in the proof of Lemma \ref{asso in e1}, 
$* \tilde{\varphi}(e_{0}, e_{1}, e_{3}, \cdot)_{e_{2}} = 0$ is equivalent to 
$A_{3} (A_{2}^{2} - A_{3}^{2})=0$. 
Hence we have 
\begin{align}
(c, s, A_{2}, A_{3}, B_{3}, A_{4}, B_{4}) =& (1, 0, 1, 0, 0, 0, 0), \label{plane sol 0} \\
                                               & \left(1, 0, \frac{\sqrt{3}}{2}, \pm \frac{1}{2}, 0, 0, 0 \right). \label{plane sol 1}
\end{align}

\underline{Case (1)-(ii)-(a)}
By normalizing $e_{2}$, we have $A_{2} = 1$. Then we see 
\begin{align}
(c, s, A_{2}, A_{3}, B_{3}, A_{4}, B_{4}) = \left( \frac{\sqrt{3}}{2}, \frac{1}{2}, 1, 0, 0, 0, 0 \right). \label{plane sol 2}
\end{align}

In case
\underline{(1)-(ii)-(b)}, 
\underline{(1)-(ii)-(c)}, and 
\underline{(1)-(iii)-(b)}, 
we have the following solutions: 
\begin{align}
(c, s, A_{2}, A_{3}, B_{3}, A_{4}, B_{4}) =& \left( \frac{\sqrt{3}}{2}, \frac{1}{2}, \frac{\sqrt{3}}{2}, 
                                                  0, 0, 0, \pm \frac{1}{2} \right), \label{plane sol 3} \\
                                                 & \left( \frac{\sqrt{3}}{2}, \frac{1}{2}, \frac{1}{2}, 
                                                  0, 0, 0, \frac{\sqrt{3}}{2} \right), \label{plane sol 4} \\
                                                 & \left( \frac{1}{2}, \frac{\sqrt{3}}{2}, \frac{\sqrt{3}}{2}, 
                                                  0, 0, 0, \frac{1}{2} \right). \label{plane sol 5}
\end{align}
In case \underline{(1)-(iii)-(a)} and \underline{(1)-(iii)-(c)}, 
we have no solutions. 

The solution (\ref{plane sol 0}) corresponds to the $\mathbb{H}$-plane. 
The planes corresponding to (\ref{plane sol 2}), (\ref{plane sol 3}), (\ref{plane sol 4}), and (\ref{plane sol 5}) 
are congruent up to the ${\rm Sp}(1) {\rm Sp}(2)$-action to that of (\ref{plane sol 1}), 
which is not associative at $(e_{0}+e_{1})/\sqrt{2}$ 
since $* \tilde{\varphi}((-e_{0}+e_{1})/\sqrt{2}, (e_{2}-e_{3})/\sqrt{2}, (e_{2}+e_{3})/\sqrt{2})_{(e_{0}+e_{1})/\sqrt{2}} \neq 0$.

\underline{Case (2)}
We may assume that the first and the second entries of $e_{3}$ are zero. 
Hence we have $B_{3}s = A_{4}s = B_{4}s = 0$. 
If $s \neq 0$, we obtain the plane $V_{2}$. 
If $s = 0$, the corresponding plane is congruent up to ${\rm Sp}(2)$-action to $V_{2}$. 

\underline{Case (3)}
We may assume that the first and the second entries of $e_{2}$ and $e_{3}$ are zero. 
However, this implies that $e_{3} = 0$, which is a contradiction. 
\end{proof}

\section{Classification of homogeneous associative submanifolds}

In this section, we prove Theorem \ref{classification homog asso}. 
First, we classify compact Lie subgroups of ${\rm Sp}(1) {\rm Sp}(2)$ which have 3-dimensional orbits.  
Let $G$ be a compact connected Lie subgroup of ${\rm Sp}(1) {\rm Sp}(2)$.
Suppose that $G$ has a 3-dimensional orbit $A$. 
Since $G$ acts on $A$ as an isometry group, 
$\dim G \leq 3 \cdot (3+1)/2 = 6$ and $\dim G \neq 5$. 
(see \cite{Yano}, Chapter I\hspace{-.1em}V, Theorem 9.1). 
We only have to consider the Lie algebra $\mathfrak{g} \subset 
\mathfrak{sp}(1) \oplus \mathfrak{sp}(2)$ of $G$.

\subsection{Case $\dim \mathfrak{g} = 3$}

Suppose that $\underline{\dim \mathfrak{g} = 3}$. 
By the classification of the compact Lie algebras, 
$\mathfrak{g}$ is isomorphic to $\mathfrak{su}(2)$ or $\mathfrak{t}^{3}$, 
where $\mathfrak{t}^{3}$ is a Lie algebra of the 3-torus $T^{3}$. 
The case $\mathfrak{g} = \mathfrak{t}^{3}$ 
corresponds to 
the inclusion $T^{3} \hookrightarrow {\rm U}(1) {\rm Sp}(2) \subset {\rm U}(4)$ given by 
\begin{align} \label{T3 action} 
(e^{i \alpha}, e^{i \beta}, e^{i \gamma}) \mapsto 
{\rm diag}(e^{i(\alpha + \beta)},  e^{i(\alpha - \beta)}, e^{i(\alpha + \gamma)}, e^{i(\alpha - \gamma)}), 
\end{align}
which is a maximal torus of ${\rm Sp}(1) {\rm Sp}(2)$ 
and induces the $T^{3}$-action on $S^{7}$. 
Define the basis $\{ F_{1}, F_{2}, F_{3} \}$ of the Lie algebra $\mathfrak{t}^{3} \cong \mathbb{R}^{3}$ of $T^{3}$ by 
\begin{align} \label{basis of t3}
F_{1} = (1, 0, 0), \qquad
F_{2} = (0, 1, 0), \qquad
F_{3} = (0, 0, 1). 
\end{align}
Via the inclusion $\mathfrak{t}^{3} \hookrightarrow \mathfrak{u}(1) \oplus \mathfrak{sp}(2)$, 
$F_{1}, F_{2}, F_{3}$ correspond to
\begin{align} \label{lie alg of T3}
\left(
\begin{array}{cccc}
 i &  &  & \\
   & i &  & \\
   &  & i & \\
   &  &   & i
\end{array} 
\right), 
\left(
\begin{array}{cccc}
i&  & & \\
 & -i & & \\
 &  & 0   & \\
 &  &   & 0
\end{array} 
\right), 
\left(
\begin{array}{cccc}
0  &  &  & \\
   & 0&   & \\
   &  & i & \\
   &  &   &-i 
\end{array} 
\right), 
\end{align}
respectively.

When $\mathfrak{g} = \mathfrak{su}(2)$, we see 
that $\mathfrak{su}(2) = \mathfrak{sp}(1)_{L}$ or $\mathfrak{su}(2) \subset \mathfrak{sp}(2)_{R}$. 
Suppose that $\mathfrak{su}(2) \subset \mathfrak{sp}(2)_{R}$. 
Recall that any representation of the compact Lie group ${\rm SU}(2)$ is 
completely reducible and 
the dimension of the real irreducible representation of ${\rm SU}(2)$ is of the form 
$4k, 2l-1 (k, l \geq1)$. 
Thus we have 3 types of inclusions 
$\mathfrak{su}(2) \hookrightarrow \mathfrak{so}(5)$ given by 
\begin{align*}
&\mathfrak{su}(2) = \mathfrak{so}(3) \hookrightarrow \mathfrak{so}(5),\\
&\mathfrak{su}(2) \hookrightarrow \mathfrak{so}(4) \hookrightarrow \mathfrak{so}(5),\\
&\mathfrak{su}(2) \hookrightarrow \mathfrak{so}(5) \mbox {: irreducibly}. 
\end{align*}
The identification $\mathfrak{sp}(2) = \mathfrak{so}(5)$ 
induces three types of inclusions ${\rm SU}(2) \hookrightarrow {\rm Sp}(2)$. 
Hence we have 
the following four types of inclusions ${\rm SU}(2) \hookrightarrow {\rm Sp}(1) {\rm Sp}(2)$.

\begin{enumerate}
\item ${\rm SU}(2) = {\rm Sp}(1)_{L}$ acting on $S^{7}$ by (\ref{right sp1 action}), 
\item The inclusion ${\rm SU}(2) \hookrightarrow {\rm Sp}(2)$ given by 
\begin{align} \label{diag SU2} 
\left(
\begin{array}{cc}
a & -\overline{b} \\
b & \overline{a} \\
\end{array} 
\right) 
\mapsto 
\left(
\begin{array}{cccc}
\overline{a}&  & -b & \\
               & a&     & - \overline{b} \\
\overline{b}&  &a    & \\
               &b&      & \overline{a}
\end{array} 
\right), 
\end{align}
which induces the ${\rm SU}(2)$-action on $S^{7}$. 
Define the basis $\{ E_{1}, E_{2}, E_{3} \}$ of the Lie algebra $\mathfrak{su}(2)$ of ${\rm SU}(2)$
satisfying $[E_{i}, E_{i+1}] = 2E_{i+2}$ for $i \in \mathbb{Z}/3$ by 
\begin{align} \label{basis of su2}
E_{1} = 
\left(
\begin{array}{cc}
0  & 1 \\
-1 & 0 \\
\end{array} 
\right),  \qquad
E_{2} = 
\left(
\begin{array}{cc}
0  & i \\
i   & 0 \\
\end{array} 
\right),  \qquad
E_{3} = 
\left(
\begin{array}{cc}
i  &0 \\
0 &i \\
\end{array} 
\right). 
\end{align}
Via this inclusion $\mathfrak{su}(2) \hookrightarrow \mathfrak{sp}(2)$, 
$E_{1}, E_{2}, E_{3}$ correspond to
\begin{align} \label{lie alg of diag SU2}
\left(
\begin{array}{cccc}
   &  & 1 & \\
   &  &    & 1 \\
-1&  &    & \\
   &-1&   & 
\end{array} 
\right), 
\left(
\begin{array}{cccc}
   &  & -i & \\
   &  &    & i \\
-i&  &    & \\
   &i &   & 
\end{array} 
\right), 
\left(
\begin{array}{cccc}
-i &  &  & \\
   & i&   & \\
   &  & i & \\
   &  &   &-i 
\end{array} 
\right), 
\end{align}
respectively. 
\item 
The inclusion ${\rm SU}(2) \hookrightarrow {\rm Sp}(2)$ given by 
\begin{align} \label{small SU2}
A
\mapsto 
\left(
\begin{array}{cc}
A & O_{2} \\
O_{2} & I_{2} \\
\end{array} 
\right). 
\end{align}
Via this inclusion $\mathfrak{su}(2) \hookrightarrow \mathfrak{sp}(2)$, 
$E_{1}, E_{2}, E_{3}$ correspond to
\begin{align} \label{lie alg of small SU2}
\left(
\begin{array}{ccc}
   & 1&  \\
-1&  &   \\
   &  & O_{2}  \\
\end{array} 
\right), 
\left(
\begin{array}{ccc}
   & i  &   \\
i  &   &    \\
  &   &  O_{2}
\end{array} 
\right), 
\left(
\begin{array}{ccc}
 i &  &  \\
   &-i&   \\
   &  & O_{2}
\end{array} 
\right), 
\end{align}
respectively. 
\item
The inclusion ${\rm SU}(2) \hookrightarrow {\rm Sp}(2)$ given by 
\begin{align} \label{irr SU2}
\left(
\begin{array}{cc}
a & -\overline{b} \\
b & \overline{a} \\
\end{array} 
\right) 
\mapsto 
\left(
\begin{array}{cccc}
a^{3}               & -\overline{b}^{3}              & \sqrt{3} a \overline{b}^{2} & - \sqrt{3} a^{2} \overline{b} \\
b^{3}               & \overline{a}^{3}               & \sqrt{3} \overline{a}^{2} b & \sqrt{3} \overline{a} b^{2} \\
\sqrt{3} ab^{2}  & -\sqrt{3} \overline{a}^{2} \overline{b} & \overline{a}(|a|^{2} -2|b|^{2}) & b (2|a|^{2} -|b|^{2}) \\
\sqrt{3} a^{2}b  & \sqrt{3} \overline{a} \overline{b}^{2}   & -\overline{b} (2|a|^{2} -|b|^{2}) & a(|a|^{2} -2|b|^{2})
\end{array} 
\right), 
\end{align}
which induces the ${\rm SU}(2)$-action on $S^{7}$. 
This action is an irreducible representation of ${\rm SU}(2)$ on $\mathbb{C}^{4}$. 
This is the induced action of ${\rm SU}(2)$ on $V_{3} = \mathbb{C}^{4}$ 
from the standard action on $\mathbb{C}^{2}$, 
where we use the notation of Lemma \ref{Fourier_S3}. 
Via $\mathfrak{su}(2) \hookrightarrow \mathfrak{sp}(2)$, 
$E_{1}, E_{2}, E_{3}$ correspond to
\begin{align} \label{lie alg of irr SU2}
\left(
\begin{array}{cccc}
            &             &              & \sqrt{3} \\
            &             & -\sqrt{3} &             \\
            &  \sqrt{3} &             & -2         \\
-\sqrt{3} &            & 2           & 
\end{array} 
\right), 
\left(
\begin{array}{cccc}
            &             &              & \sqrt{3}i \\
            &             & \sqrt{3}i  &             \\
            &  \sqrt{3}i &             & 2i         \\
\sqrt{3}i &              & 2i           & 
\end{array} 
\right), 
\left(
\begin{array}{cccc}
3i &    &   & \\
   & -3i&   & \\
   &    & -i & \\
   &    &    &i 
\end{array} 
\right), 
\end{align}
respectively. 
\end{enumerate}

\subsection{Case $\dim \mathfrak{g} = 4$} \label{Lie alg classification dim4}
By the classification of the compact Lie algebras, 
$\mathfrak{g}$ is isomorphic to $\mathfrak{su}(2) \oplus \mathbb{R}$. 
Since 
the inclusions 
$\mathfrak{su}(2) \hookrightarrow \mathfrak{sp} (1) \oplus \mathfrak{sp} (2)$ are classified, 
we have to find the 1-dimensional Lie subalgebras 
which commute with $\mathfrak{su} (2)$. 
Set 
\begin{align*}
Z(\mathfrak{su} (2)) = 
\{ X \in \mathfrak{sp} (1) \oplus \mathfrak{sp} (2); [X, Y] = 0 
\mbox{ for any }Y \in \mathfrak{su}(2) \}. 
\end{align*}

First consider the case $\mathfrak{su} (2) = \mathfrak{sp} (1)_{L}$.
Then we have 
$Z(\mathfrak{su}(2)) = \mathfrak{sp}(2)_{R}$. 
Take any 1-dimensional subspace $\mathfrak{k} \subset \mathfrak{sp}(2)_{R}$ 
and suppose that 
$G$ is the Lie subgroup of ${\rm Sp}(1) {\rm Sp}(2)$ whose Lie algebra is $\mathfrak{su}(2) \oplus \mathfrak{k}$.
Since the ${\rm Sp}(1)_{L}$-action commutes with the ${\rm Sp}(2)_{R}$-action, 
the $G$-orbit through $p \in S^{7}$ 
should be contained in ${\rm Sp}(1) \cdot p$ 
so that 
it is 3-dimensional. 
Thus this case is reduced to that of  (\ref{right sp1 action}).

Next, suppose that $\mathfrak{su}(2) \subset \mathfrak{sp} (2)$ is induced from (\ref{diag SU2}). 
In this case, we have
$
Z(\mathfrak{su} (2)) = \mathfrak{sp}(1)_{L} \oplus (\mathbb{R} {\rm diag}(i, -i, i, -i))_{R}.
$
The Lie subgroup $G \subset {\rm Sp}(2)$ 
whose Lie algebra is $(\mathfrak{su}(2) \oplus \mathbb{R} {\rm diag}(i, -i, i, -i))_{R}$ 
is ${\rm U}(2)$ 
whose restriction to ${\rm SU}(2)$ is given by (\ref{diag SU2}). 
This ${\rm U}(2)$ action has the same orbits as the ${\rm SU}(2)$-action. 
The new 3-dimensional orbits do not appear from $\mathfrak{sp}(1)_{L}$, 
and this case is reduced to that of  (\ref{diag SU2}).

Suppose that $\mathfrak{su}(2) \subset \mathfrak{sp} (2)$ is induced from (\ref{small SU2}). 
In this case, we have
$
Z(\mathfrak{su} (2)) = \mathfrak{sp}(1)_{L} \oplus 
\left(
\begin{array}{cc}
O_{2}  & \\
        & \mathfrak{su}(2) 
\end{array} 
\right)_{R}.
$
This case is also reduced to that of  (\ref{small SU2}) in the same way.

Suppose that $\mathfrak{su}(2) \subset \mathfrak{sp} (2)$ is induced from (\ref{irr SU2}). 
In this case, we have
$Z(\mathfrak{su} (2)) = \mathfrak{sp}(1)_{L}$. 
This case is also reduced to that of  (\ref{irr SU2}) in the same way.

\subsection{Case $\dim \mathfrak{g} = 6$}
By the classification of the compact Lie algebras, 
$\mathfrak{g}$ is isomorphic to 
$\mathfrak{su}(2) \oplus \mathfrak{t}^{3}$ or 
$\mathfrak{su}(2) \oplus \mathfrak{su}(2)$.  
When $\mathfrak{g} \cong \mathfrak{su}(2) \oplus \mathfrak{t}^{3}$, 
we have $\mathfrak{g} \cong \mathfrak{t}^{1}_{L} \oplus (\mathfrak{su}(2) \oplus \mathfrak{t}^{2})_{R}$. 
Since  
there are no 2-dimensional commutative Lie subalgebras of $\mathfrak{sp}(2)$
which commute with $\mathfrak{su}(2)$ by Section \ref{Lie alg classification dim4}, 
this case does not occur. 

When $\mathfrak{g} \cong \mathfrak{su}(2) \oplus \mathfrak{su}(2)$, 
we have 
$G = {\rm Sp}(1)_{L} \cdot {\rm SU}(2)_{R}$ or 
$
\left(
\begin{array}{cc}
{\rm SU}(2)  & \\
                 &{\rm SU}(2) 
\end{array} 
\right)_{R}
$, 
which reduces to the case above.

Thus we only have to consider the orbits of 
(\ref{T3 action}), (\ref{right sp1 action}), (\ref{diag SU2}), (\ref{small SU2}), and 
(\ref{irr SU2}).


\subsection{$T^{3}$-orbits}

We classify associative submanifolds which are orbits of $T^{3}$ acting on $S^{7}$ as (\ref{T3 action}). 

\begin{prop}
Up to the ${\rm Sp}(1) {\rm Sp}(2)$-action, 
$T^{3} \cdot \frac{1}{2} {}^t\! (1, 1, 1, i)$ is the unique associative submanifold 
in the squashed $S^{7}$ 
which is an orbit of the $T^{3}$-action. 
\end{prop}

\begin{rem}
The associative orbit $A_{1} = T^{3} \cdot \frac{1}{2} {}^t\! (1, 1, 1, i)$ is 
the Hopf lift of a $I_{1}^{'}$-holomorphic curve in $\mathbb{C}P^{3}$, 
where $I_{1}^{'}$ is defined by (\ref{almcpxstr CP3}).
We have 
\begin{align*}
A_{1} = 
\left \{
{}^t\! (z_{1}, z_{2}, z_{3}, z_{4}) \in S^{7}; 
\begin{array}{c}
|z_{1}|= |z_{2}| = |z_{3}| = |z_{4}|,  \\
{\rm Re}(z_{1}z_{2} \overline{z}_{3} \overline{z}_{4}) = 0, 
{\rm Im}(z_{1}z_{2} \overline{z}_{3} \overline{z}_{4}) < 0 \\
\end{array} 
\right \}, 
\end{align*}
which is a special Legendrian given in \cite{Harvey Lawson}
via ${}^t\! (z_{1}, z_{2}, z_{3}, z_{4}) \mapsto {}^t\! (z_{1}, z_{2}, \overline{z}_{3}, \overline{z}_{4})$. 
The inclusion (\ref{T3 action}) induces the metric 
$\frac{3}{5} (F^{1})^{2} + \frac{27}{50} (F^{2})^{2} + \frac{27}{50} (F^{3})^{2}$, 
where $\{ F^{i} \}$ is the dual of $\{ F_{i} \}$. 
\end{rem}

\begin{proof}
Fix $p_{0} = {}^t\! (z_{1}, z_{2}, z_{3}, z_{4}) \in S^{7}$ and set $A= T^{3} \cdot p_{0}$. 
Then the tangent space $T_{p_{0}} A$ is spanned by 
the vectors $F_{i}^{*}$ generated by $F_{i}$ in (\ref{basis of t3}):
\begin{align*}
(F_{1}^{*})_{p_{0}} &= i {}^t\! (z_{1}, z_{2}, z_{3}, z_{4}) = -\xi_{1}, \\
(F_{2}^{*})_{p_{0}} &= i {}^t\! (z_{1}, - z_{2}, 0, 0), \\
(F_{3}^{*})_{p_{0}} &= i {}^t\! (0, 0, z_{3}, -z_{4}). 
\end{align*}
By Lemma \ref{equiv condi asso coasso}, 
we consider the condition $* \tilde{\varphi}(F_{1}^{*}, F_{2}^{*}, F_{3}^{*}, \cdot)|_{T_{p_{0}}S^{7}} = 0$. 
We easily see that $-i(F_{1}^{*}) * \tilde{\varphi} = (3^{4}/5^{3})
{\rm Im}((\eta_{2} - i \eta_{3}) \wedge d(\eta_{2}+ i \eta_{3}))$.
From Lemma \ref{xi eta S7}, we have 
\begin{align*}
(\eta_{2} + i \eta_{3})(F_{2}^{*}) &= 2i z_{1} z_{2}, 
&(\eta_{2} + i \eta_{3})(F_{3}^{*}) &= 2i z_{3} z_{4}, \\
d(\eta_{2} + i \eta_{3})(F_{2}^{*}, \cdot) &= -2i d(z_{1} z_{2}), 
&d(\eta_{2} + i \eta_{3})(F_{3}^{*}, \cdot) &= -2i d(z_{3} z_{4}),  
\end{align*}
which implies that 
the condition $* \tilde{\varphi}(F_{1}^{*}, F_{2}^{*}, F_{3}^{*}, \cdot)|_{T_{p_{0}}S^{7}} = 0$ is equivalent to 
$d ({\rm Im}(z_{1} z_{2} \overline{z}_{3} \overline{z}_{4})) = 0$. 
The restriction of this form to $TS^{7}$ is given by  
$d ({\rm Im}(z_{1} z_{2} \overline{z}_{3} \overline{z}_{4})) 
- d ({\rm Im}(z_{1} z_{2} \overline{z}_{3} \overline{z}_{4}))(r \frac{\partial}{\partial r}) \frac{dr}{r}
= {\rm Re}(\sum_{j = 1}^{4} \zeta_{j} dz_{j})$, where 
$r \frac{\partial}{\partial r}$ is a position vector, $\frac{dr}{r}$ is its dual, and 
\begin{align*}
\left(
\begin{array}{c}
\zeta_{1}  \\ \zeta_{2} \\ \zeta_{3} \\ \zeta_{4} 
\end{array} 
\right)
=
\left(
\begin{array}{c}
-i z_{2} \overline{z}_{3} \overline{z}_{4} \\
-i z_{1} \overline{z}_{3} \overline{z}_{4}\\
i \overline{z}_{1} \overline{z}_{2} z_{4}\\
i \overline{z}_{1} \overline{z}_{2} z_{3}\\
\end{array} 
\right)
- 
4 {\rm Im}(z_{1} z_{2} \overline{z}_{3} \overline{z}_{4})
\left(
\begin{array}{c}
\overline{z}_{1}  \\ \overline{z}_{2} \\ \overline{z}_{3} \\ \overline{z}_{4} 
\end{array} 
\right).
\end{align*}
Thus we see that 
the condition $* \tilde{\varphi}(F_{1}^{*}, F_{2}^{*}, F_{3}^{*}, \cdot)|_{T_{p_{0}}S^{7}} = 0$ is 
equivalent to 
$\zeta_{j} (p_{0}) = 0$ for  $j=1, \cdots, 4.$ 
On the other hand, setting 
\begin{align*}
\Sigma = \{ {}^t\! (x_{1}, x_{2}, x_{3}, x_{4} + i y_{4})\in S^{7} \subset \mathbb{C}^{4}; 
x_{j}, y_{4}\in \mathbb{R}, x_{1}, x_{2}, x_{3} \geq 0 \}, 
\end{align*}
we have $S^{7} = T^{3} \cdot \Sigma$. 
Hence we may assume that $p_{0} \in \Sigma$ 
and $x_{1}, x_{2}, x_{3} \neq 0$ 
so that $T^{3} \cdot p_{0}$ is 3-dimensional. 
Then we can solve $\zeta_{j} = 0$ easily to obtain 
\begin{align*}
x_{1} = x_{2} = x_{3} = 1/2, \qquad x_{4} = 0, \qquad y_{4} = \pm 1/2. 
\end{align*} 
The $T^{3}$-orbit through ${}^t\! (1,1,1,i)/2$ is mapped to 
that through ${}^t\! (1,1,1,-i)/2$ by 
$
\left(
\begin{array}{cc}
I_{2} & 0 \\
0 & K \\
\end{array} 
\right) 
\in {\rm Sp}(2), $
where
$
K =  
\left(
\begin{array}{cc}
0 & -i \\
-i & 0 \\
\end{array} 
\right),
$
and 
we obtain the statement. 
\end{proof}


\subsection{${\rm SU}(2)$-orbits}

We consider the ${\rm SU}(2)$-orbits of 
(\ref{right sp1 action}), (\ref{diag SU2}), (\ref{small SU2}), or (\ref{irr SU2}). 
First, we introduce a useful lemma to study associative orbits.

\begin{lem} (\cite{Mashimo} Lemma 5.6.) \label{on frame SU2 orbit}
Let $(V, \rho)$ be an orthogonal representation of ${\rm SU}(2)$, 
$\langle \cdot, \cdot \rangle$ be an ${\rm SU}(2)$-invariant
inner product on $V$, 
and $S_{1} \subset V$ be the unit sphere. 
Let $M = {\rm SU}(2) \cdot p$ be a 3-dimensional orbit through $p\in S_{1}$.
Define the function $\lambda_{j}: M \rightarrow \mathbb{R}$ for $j = 1, 2, 3$ by 
\begin{align*}
\lambda_{j} = \langle (\rho_{*}(E_{j}))^{*}, (\rho_{*}(E_{j}))^{*} \rangle|_{M}, 
\end{align*}
where 
$\{ E_{j} \}$ is a basis of ${\rm su}(2)$ satisfying $[E_{i}, E_{i+1}] = 2 E_{i+2}$ for $i \in \mathbb{Z}/3$ and 
$(\rho_{*}(E_{j}))^{*}$ is a vector field on $V$ generated by $\rho_{*}(E_{j}) \in \mathfrak{gl}(V)$. 
Denote by $\{ E^{j} \}$ the dual 1-form on $M$ of $\{ (\rho_{*}(E_{j}))^{*}|_{M} \}$. 
Then there exists $g \in {\rm SU}(2)$, the induced metric $\langle \cdot, \cdot \rangle|_{M}$ 
is described as 
\begin{align} \label{met SU2 orbit}
\langle \cdot, \cdot \rangle|_{M} = \sum_{j=1}^{3} \lambda_{j} (E^{j})^{2}, 
\end{align}
at $g \cdot p \in M$. Moreover, 
$(M, \langle \cdot, \cdot \rangle|_{M})$ 
is a space of constant curvature $k$ if and only if 
$\lambda_{1} = \lambda_{2} = \lambda_{3} = 1/k$. 
\end{lem}

\begin{rem}(\cite{Mashimo} Remark 5.4.) \label{permute frame}
There exists $g' \in {\rm SU}(2)$ satisfying (\ref{met SU2 orbit}) and 
$\lambda_{1}(g') = \lambda_{a}(g), \lambda_{2}(g') = \lambda_{b}(g), \lambda_{3}(g') = \lambda_{c}(g)$, 
where 
$\{ a, b, c\}$ is any permutation of $\{1, 2, 3 \}$. 
Thus 
we can ``permute" the $\lambda_{j}$. 
\end{rem}

\subsubsection{${\rm SU}(2)$-orbits 1}

If an ${\rm SU}(2)$-action is given by (\ref{right sp1 action}), 
the orbit is the intersection of a quaternionic plane and $S^{7}$, 
which is an obvious totally geodesic associative submanifold.

\subsubsection{${\rm SU}(2)$-orbits 2}

Consider the ${\rm SU}(2)$-action given by (\ref{diag SU2}). 
Let $A$ be an  ${\rm SU}(2)$-orbit through $p_{0} = {}^t\! (z_{1}, z_{2}, z_{3}, z_{4})$. 
Then the tangent space to $A$ at $p_{0}$ is spanned by 
the vectors $E_{i}^{*}$ generated by $E_{i}$ in (\ref{basis of su2}):
\begin{align*}
(E_{1}^{*})_{p_{0}} &=  {}^t\! (z_{3}, z_{4}, -z_{1}, -z_{2}), \\
(E_{2}^{*})_{p_{0}} &=  i {}^t\! (- z_{3}, z_{4}, - z_{1}, z_{2}), \\
(E_{3}^{*})_{p_{0}} &=  i {}^t\! (- z_{1}, z_{2}, z_{3}, -z_{4}). 
\end{align*}
We easily see that 
$g(E_{i}^{*}, E_{j}^{*})_{p_{0}} = \delta_{i j}$, 
where $g$ is the standard metric on $S^{7}$. 
Then from Lemma \ref{on frame SU2 orbit}, 
 $A$ is a constant curvature 1 submanifold of $(S^{7}, g)$. 
Thus $A$ is of the form $V \cap S^{7}$, 
where $V \subset \mathbb{R}^{8}$ is a 4-plane. 
These associative submanifolds are classified by Theorem \ref{plane asso}.


\subsubsection{${\rm SU}(2)$-orbits 3}

Consider the ${\rm SU}(2)$-action given by (\ref{small SU2}). 
Let $A$ be an  ${\rm SU}(2)$-orbit through $p_{0} = {}^t\! (z_{1}, z_{2}, z_{3}, z_{4})$. 
By the ${\rm SU}(2)$-action, 
we may assume that 
$p_{0} = {}^t\! (x_{1}, 0, z_{3}, z_{4})$ 
where $x_{1} >0, z_{3}, z_{4} \in \mathbb{C}$.  
Then the tangent space to $A$ at $p_{0}$ is spanned by 
the vectors $E_{i}^{*}$ generated by $E_{i}$ in (\ref{basis of su2}):
\begin{align*}
(E_{1}^{*})_{p_{0}} &=  {}^t\! (0, -x_{1}, 0, 0), \\
(E_{2}^{*})_{p_{0}} &=  {}^t\! (0, i x_{1}, 0, 0), \\
(E_{3}^{*})_{p_{0}} &=  {}^t\! (i x_{1}, 0, 0, 0). 
\end{align*}
We compute 
\begin{align*}
(\eta_{i} (E_{j}^{*})) &= 
\left(
\begin{array}{ccc}
0          & 0          & -x_{1}^{2}\\
-x_{1}^{2}& 0          & 0          \\
0          & -x_{1}^{2}& 0 \\
\end{array} 
\right), \\
\left(
\begin{array}{c}
d \eta_{j} (E_{1}^{*}, E_{2}^{*}) \\
d \eta_{j} (E_{1}^{*}, E_{3}^{*}) \\
d \eta_{j} (E_{2}^{*}, E_{3}^{*}) \\
\end{array} 
\right)
&=
\left(
\begin{array}{ccc}
x_{1}^{2} & 0          & 0\\
0          & 0          & -x_{1}^{2} \\
0          & -x_{1}^{2}& 0 \\
\end{array} 
\right), 
\end{align*}
\begin{align*}
&\left(
i( E_{i}^{*}) d \eta_{1}
\right)
=
2 x_{1} 
\left(
\begin{array}{c}
{\rm Im}(d z_{2})  \\
{\rm Re}(d z_{2})  \\
{\rm Re}(d z_{1}) 
\end{array} 
\right), 
\qquad
\left(
i( E_{i}^{*}) d (\eta_{2} + i \eta_{3} )
\right)
=
2 x_{1} 
\left(
\begin{array}{c}
 -dz_{1}  \\ 
i dz_{1}  \\
-i dz_{2} 
\end{array} 
\right),\\
&\sum_{i=1}^{3} d \eta_{i} (E_{1}^{*}, E_{2}^{*}, E_{3}^{*}, \cdot) = 12 x_{1}^{3} d x_{1}, \qquad
d (\eta_{123}) = 2 x_{1}^{5} d x_{1}.
\end{align*}
Since 
$* \tilde{\varphi} = \frac{27}{25} (\frac{1}{8} \sum_{i=1}^{3} (d \eta_{i})^{2} + \frac{4}{5} d (\eta_{123}))$, 
we obtain
$* \tilde{\varphi}  (E_{1}^{*}, E_{2}^{*}, E_{3}^{*}, \cdot) = \frac{5}{54} x_{1}^{3} (15 + 16 x_{1}^{2}) d x_{1}$. 
The restriction of $d x_{1}$ to $T S^{7}$ is given by 
\begin{align*}
d x_{1} - d x_{1} \left(r \frac{\partial}{\partial r} \right) \frac{d r}{r} = d x_{1} 
- x_{1} \left( x_{1} d x_{1} + {\rm Re} (z_{3} d z_{3} + z_{4} d z_{4}) \right),
\end{align*}
where 
$r \frac{\partial}{\partial r}$ is a position vector and  $\frac{dr}{r}$ is its dual. 
This implies that 
 $* \tilde{\varphi}(E_{1}, E_{2}, E_{3}, \cdot)|_{T_{p_{0}}S^{7}} = 0$ is 
equivalent to 
$x_{1} = 1, z_{3}=z_{4} = 0$, 
and the resulting associative submanifold is 
$\{ (z_{1}, z_{2}, 0, 0) \in \mathbb{C}^{4}; |z_{1}|^{2} + |z_{2}|^{2} = 1\}$.


\subsubsection{${\rm SU}(2)$-orbits 4}

For the ${\rm SU}(2)$-action given by (\ref{irr SU2}), 
we obtain the following. 

\begin{prop} \label{classification irr SU2}
Let $A$ be an associative submanifold in the squashed $S^{7}$ 
which is an orbit of the ${\rm SU}(2)$-action given in (\ref{irr SU2}). 
Then up to the ${\rm Sp}(1) {\rm Sp}(2)$-action, 
\begin{align*}
A = A_{2}:= {\rm SU}(2) \cdot {}^t\! (1, 0, 0, 0) \quad
\mbox{ or } \quad
A_{3} := {\rm SU}(2) \cdot {}^t\! (0, 0, 1, 0).  
\end{align*}
\end{prop}

\begin{rem}
The associative orbit $A_{2}$ is 
the Hopf lift of 
a horizontal holomorphic curve 
\begin{align*}
\{ [a^{3}: b^{3}: \sqrt{3}ab^{2}: \sqrt{3}a^{2}b] \in \mathbb{C}P^{3}; a, b \in \mathbb{C}, |a|^{2} + |b|^{2} = 1 \}
\end{align*}
in $\mathbb{C}P^{3}$. This is 
a degree 3 $\mathbb{C}P^{1}$ in $\mathbb{C}P^{3}$ of the constant curvature 
called the Veronese curve. 
The associative orbit $A_{3}$ is 
the Hopf lift of 
a null-torsion $I_{1}^{'}$-holomorphic curve in $\mathbb{C}P^{3}$, 
which is defined in Definition \ref{null-torsion def}. 
The inclusion (\ref{irr SU2}) induces 
$\tilde{g}|_{A_{2}} = \frac{27}{25} (5 (E^{1})^{2} + 5 (E^{2})^{2} + 3 (E^{3})^{2})$ 
and $\tilde{g}|_{A_{3}} = \frac{9}{25} (19 (E^{1})^{2} + 19 (E^{2})^{2} + (E^{3})^{2})$,  
where we use the notation of Lemma \ref{on frame SU2 orbit}. 
\end{rem}

\begin{rem} 
Set $A_{2}(a, b):= {\rm SU}(2) \cdot {}^t\! (a, b, 0, 0)$ 
and $A_{3}(a, b):= {\rm SU}(2) \cdot {}^t\! (0, 0, a, b)$ 
for $a, b \in \mathbb{C}, |a|^{2} + |b|^{2} = 1$. 
Then by the action of $a +bj \in {\rm Sp}(1)_{L}$, 
$A_{j}$  is congruent to $A_{j}(a, b) (j = 2, 3)$. 
Via ${}^t\! (z_{1}, z_{2}, z_{3}, z_{4}) \mapsto {}^t\! (z_{1}, z_{4}, z_{3}, z_{2})$, 
$A_{2}(\frac{1}{\sqrt{2}}, \frac{1}{\sqrt{2}})$ is  
special Legendrian given by \cite{Marshall}. 
\end{rem}

\begin{proof}[Proof of Proposition \ref{classification irr SU2}]
Let $A$ be an  ${\rm SU}(2)$-orbit through $p_{0} = {}^t\! (z_{1}, z_{2}, z_{3}, z_{4})$. 
Then the tangent space to $A$ at $p_{0}$ is spanned by 
the vectors $E_{i}^{*}$ generated by $E_{i}$ in (\ref{basis of su2}):
\begin{align*}
(E_{1}^{*})_{p_{0}} &=  {}^t\! (\sqrt{3} z_{4}, -\sqrt{3} z_{3}, \sqrt{3} z_{2} -2 z_{4}, -\sqrt{3} z_{1} + 2 z_{3}),\\ 
(E_{2}^{*})_{p_{0}} &=  {}^t\! (\sqrt{3} i z_{4}, \sqrt{3} i z_{3}, \sqrt{3} i z_{2} +2i z_{4}, \sqrt{3}i z_{1} + 2i z_{3}), \\
(E_{3}^{*})_{p_{0}} &=  {}^t\! (3i z_{1}, -3i z_{2}, -i z_{3}, i z_{4}). 
\end{align*}
Since ${\rm SU}(2) \subset {\rm Sp}(2)$-action preserves $\eta_{j}$, 
we have $L_{E_{j}^{*}} \eta_{i} = 
d \eta_{i} (E_{j}^{*}, \cdot) + d (\eta_{i}(E_{j}^{*})) = 0$. 
Then by the equation 
$
[E_{j}^{*}, E_{j+1}^{*}] = -2 E_{j+2}^{*}$ 
for 
$j \in \mathbb{Z}/3, 
$
we have 
\begin{align} \label{computation su2 4}
\sum_{i=1}^{3} (d \eta_{i})^{2} (E_{1}^{*}, E_{2}^{*}, E_{3}^{*}, \cdot) \nonumber
&= 
2 
\sum_{i,j=1}^{3} d (\eta_{i} (E_{j}^{*})^{2}), \\
d (\eta_{123}) (E_{1}^{*},E_{2}^{*},E_{3}^{*}, \cdot)
&=
-
d (\eta_{123} (E_{1}^{*},E_{2}^{*},E_{3}^{*})). 
\end{align}
We compute 
\begin{align*}
\eta_{1} (E_{2}^{*}) + i \eta_{1} (E_{1}^{*}) 
&= -2 \sqrt{3} (\overline{z}_{1} z_{4} + z_{2} \overline{z}_{3}) -4 z_{3} \overline{z}_{4}, \\
\eta_{1} (E_{3}^{*})
&= -3 |z_{1}|^{2} + 3 |z_{2}|^{2} +  |z_{3}|^{2} -  |z_{4}|^{2}, \\
(\eta_{2} + i \eta_{3}) (E_{1}^{*}) 
&= 2 \sqrt{3} (z_{1} z_{3} + z_{2} z_{4}) -2 ( z_{3}^{2} + z_{4}^{2}), \\
(\eta_{2} + i \eta_{3}) (E_{2}^{*}) 
&= 2 \sqrt{3} i (- z_{1} z_{3} +  z_{2} z_{4}) +2i (- z_{3}^{2} + z_{4}^{2}), \\
(\eta_{2} + i \eta_{3}) (E_{3}^{*}) 
&=
6i z_{1} z_{2} -2i z_{3} z_{4}, 
\end{align*}
Then 
we have
$
\sum_{i,j=1}^{3} \eta_{i} (E_{j}^{*})^{2} = 9 
$
and 
$\sum_{i=1}^{3} (d \eta_{i})^{2} (E_{1}^{*}, E_{2}^{*}, E_{3}^{*}, \cdot) = 0$ 
by (\ref{computation su2 4}). 
Since 
$* \tilde{\varphi} = \frac{27}{25} (\frac{1}{8} \sum_{i=1}^{3} (d \eta_{i})^{2} + \frac{4}{5} d (\eta_{123}))$, 
the condition $* \tilde{\varphi}(E_{1}^{*}, E_{2}^{*}, E_{3}^{*}, \cdot) = 0$ is equivalent to 
\begin{align*}
d (\det M) =0, 
\end{align*}
where $M = (\eta_{i} (E_{j}^{*}))$.

Now, we use Lemma \ref{on frame SU2 orbit}. 
We may assume that 
$\{ E_{1}^{*}, E_{2}^{*}, E_{3}^{*} \}$ are mutually orthogonal at 
$p_{0} = {}^t\! (z_{1}, z_{2}, z_{3}, z_{4})$ with respect to $g$. 
Then we have 
\begin{align} \label{SU2-2 orthog}
z_{1} \overline{z}_{4} - \overline{z}_{2} z_{3} =0, \qquad 
{\rm Im} (z_{1} \overline{z}_{3} + \overline{z}_{2} z_{4}) =0. 
\end{align}
Setting 
\begin{align*}
\lambda_{1} &= |E_{1}^{*}|^{2} = 4 (|z_{3}|^{2} + |z_{4}|^{2}) 
- 4 \sqrt{3} {\rm Re}(z_{1} \overline{z}_{3} + \overline{z}_{2} z_{4}) + 3, \\
\lambda_{2} &= |E_{2}^{*}|^{2} = 4 (|z_{3}|^{2} + |z_{4}|^{2}) 
+ 4 \sqrt{3} {\rm Re}(z_{1} \overline{z}_{3} + \overline{z}_{2} z_{4}) + 3, \\
\lambda_{3} &= |E_{3}^{*}|^{2} = 8 (|z_{1}|^{2} + |z_{2}|^{2}) + 1.
\end{align*}
We consider the following two cases 
as the proof of Lemma 5.7 in \cite{Lotay3}: 
\begin{center}
(1) all of the $\lambda_{j}$ are distinct, \qquad
(2) at least two of the $\lambda_{j}$ are equal. 
\end{center}
Consider the case (1). 
Since we can permute the $\lambda_{j}$ by Remark \ref{permute frame}, 
we may assume that $\lambda_{3} < \lambda_{1} < \lambda_{2}$. 
The inequality $\lambda_{1} < \lambda_{2}$ implies that 
${\rm Re}(z_{1} \overline{z}_{3} + z_{2} \overline{z}_{4}) > 0$. 
Thus we have $(z_{1}, z_{2}), (z_{3}, z_{4}) \neq 0$. 
From (\ref{SU2-2 orthog}), there exists $\mu \in \mathbb{R}$ satisfying 
\begin{align}
z_{3} = \mu z_{1}, \qquad z_{4} = \mu z_{2}.           \label{SU2_2 mu}
\end{align}
Note that $\lambda_{3} < \lambda_{1}$ is equivalent to $\mu > \sqrt{3}$.
Moreover, 
since the ${\rm Sp}(1)_{L}$-action commutes the ${\rm Sp}(2)_{R}$-action 
and 
$
{}^t\! (z_{1}, z_{2}, \mu z_{1}, \mu z_{2}) 
$
is mapped to 
$\frac{1}{\sqrt{\mu^{2}+1}} {}^t\! (1, 0, \mu, 0)$ 
by 
$(\overline{z}_{1} - z_{2}j)/ \sqrt{ |z_{1}|^{2} + |z_{2}|^{2}} \in {\rm Sp}(1)_{L}$, 
we may assume that 
$p_{0} = \frac{1}{\sqrt{\mu^{2}+1}} {}^t\! (1, 0, \mu, 0)$. 

Set  
$v = {}^t\! (- \mu, 0, 1, 0) \in T_{p_{0}} S^{7}$. Then we compute 
\begin{align*}
M_{p_{0}} &= 
\frac{1}{\mu^{2}+1} 
\left(
\begin{array}{ccc}
0                                & 0                                 & \mu^{2} -3   \\
2 \mu (- \mu + \sqrt{3}) & 0                                & 0 \\
0                                & -2 \mu ( \mu + \sqrt{3}) & 0 \\
\end{array} 
\right), \\
(v(M))_{p_{0}}
&=
\frac{1}{\sqrt{\mu^{2}+1}} 
\left(
\begin{array}{ccc}
0                                                 & 0                                               & 8 \mu  \\
-2 (\sqrt{3} \mu -1)(\mu + \sqrt{3})  & 0                                               & 0 \\
0                                                 & 2 (\sqrt{3} \mu +1) (\mu - \sqrt{3}) & 0 \\
\end{array} 
\right), 
\end{align*}
where $v(M)$ is the derivative of $M$ with respect to $v$. 
Then we have 
\begin{align*}
d (\det M)_{p_{0}} (v) &= \det M_{p_{0}} \cdot {\rm tr} (v(M) M^{-1})_{p_{0}} \\
&=
24 \mu  (\mu^{2}-3) (3 \mu^{2}-1) (\mu^{2}+1)^{-5/2} > 0.
\end{align*}
Thus we have no associative ${\rm SU}(2)$-orbits in the case (1).

Next, consider the case (2). 
We may assume that $\lambda_{1} = \lambda_{2}$ by Remark \ref{permute frame}.
Then we have ${\rm Re}(z_{1} \overline{z}_{3} + z_{2} \overline{z}_{4}) = 0$, 
and (\ref{SU2-2 orthog}) implies that 
\begin{align*}
z_{1} \overline{z}_{4} - \overline{z}_{2} z_{3} = 0, \qquad
z_{1} \overline{z}_{3} + \overline{z}_{2} z_{4} = 0.
\end{align*}
Thus, 
\begin{align*}
z_{1} z_{2} \overline{z}_{3} \overline{z}_{4} &= |z_{2} z_{3}|^{2} = - |z_{2} z_{4}|^{2} = 0, \\
\overline{z}_{1} \overline{z}_{2} z_{3} z_{4} &= |z_{1} z_{4}|^{2} = - |z_{1} z_{3}|^{2} = 0. 
\end{align*}
We deduce that 
either 
$z_{1} = z_{2} = 0$ or $z_{3} = z_{4} = 0$. 
Since ${}^t\! (z_{1}, z_{2}, 0, 0)$
(resp. ${}^t\! (0, 0, z_{3}, z_{4})$)
is mapped to 
${}^t\! (1, 0, 0, 0)$ (resp. ${}^t\! (0, 0, 1, 0)$)
by $\overline{z}_{1} - z_{2} j$(resp. $\overline{z}_{3} - z_{4} j$) $\in {\rm Sp}(1)_{L}$, 
we only have to consider at $p_{0} = {}^t\! (1, 0, 0, 0)$ or ${}^t\! (0, 0, 1, 0)$. 

At $p_{0} = {}^t\! (1, 0, 0, 0)$, we have 
\begin{align} \label{vfs A2}
E_{1}^{*} = {}^t\!(0, 0, 0, - \sqrt{3}), \qquad
E_{2}^{*} = {}^t\!(0, 0, 0, \sqrt{3} i), \qquad
E_{3}^{*} = {}^t\!(3i, 0, 0, 0) = -3 \xi_{1}, 
\end{align}
which are also orthogonal to each other with respect to $\tilde{g}$ 
and $\tilde{\varphi}(E_{1}^{*}, E_{2}^{*}, E_{3}^{*}) = -243/25 = 
- |E_{1}^{*}|_{\tilde{g}} |E_{2}^{*}|_{\tilde{g}} |E_{3}^{*}|_{\tilde{g}}$. 
At $p_{0} = {}^t\! (0, 0, 1, 0)$, we have 
\begin{align} \label{vfs A3}
E_{1}^{*} = {}^t\!(0, - \sqrt{3}, 0, 2), \qquad
E_{2}^{*} = {}^t\!(0, \sqrt{3} i, 0, 2i), \qquad
E_{3}^{*} = {}^t\!(0, 0, -i, 0) = \xi_{1},
\end{align}
which are also orthogonal to each other with respect to $\tilde{g}$ 
and $\tilde{\varphi}(E_{1}^{*}, E_{2}^{*}, E_{3}^{*}) = 3^{3} \cdot 19/5^{3} 
= |E_{1}^{*}|_{\tilde{g}} |E_{2}^{*}|_{\tilde{g}} |E_{3}^{*}|_{\tilde{g}}$. 
Thus we see that 
both ${\rm SU}(2)$-orbits are associative. 
\end{proof}


\section{Deformations of homogeneous associative submanifolds}

We study the deformations of homogeneous associative submanifolds 
in the squashed $S^{7}$. 
We apply the same method of \cite{K deform}
in the standard $S^{7}$.

\begin{prop} \cite{K deform}  \label{deform_asso_NP}
Let $(Y, \varphi, g)$ be a nearly parallel $G_{2}$-manifold, and 
$A^{3} \subset Y$ be an associative submanifold. 
Denote by $\nu$  the normal bundle of $A$ in $Y$ 
and by $\nabla^{\perp_{A}}$ the connection on $\nu$ induced 
by the Levi-Civita connection $\nabla$ of $(Y, g)$.  

Taking any local orthonormal frame $\{ e_{1}, e_{2}, e_{3} \}$ of $TA$, 
define the operator 
$D : C^{\infty}(A, \nu) \rightarrow C^{\infty}(A, \nu)$ by 
\begin{align*}
D \psi := \sum_{i = 1}^{3} e_{i} \times \nabla^{\perp_{A}}_{e_{i}} \psi. 
\end{align*}
Then the vector space of all infinitesimal associative deformations 
of $A^{3} \hookrightarrow Y$ is identified with 
$
\{ \psi \in C^{\infty}(A, \nu) ; D \psi = -\psi \}. 
$
\end{prop}

Thus to compute the dimensions of the infinitesimal deformation spaces, 
we only have to know $\nabla^{\perp_{A}}$ and $\times$. 
The next lemma is useful for the computation.

\begin{lem} \label{frame for deform}
Let $\{ e_{1}, e_{2}, e_{3} \}$ be 
the local oriented orthonormal frame of $TA$ satisfying $e_{3} = e_{1} \times e_{2}$. 
Choose a local normal vector field $V_{1}$ with $|V_{1}| = 1$. 

Set $V_{2} = e_{1} \times V_{1}, V_{3} = e_{2} \times V_{1}, V_{4} = - e_{3} \times V_{1}$. 
Then $\{ V_{1}, V_{2}, V_{3}, V_{4} \}$ is a local orthonormal frame of $\nu$ 
satisfying 
\begin{align*}
\varphi = e^{123} + e^{1} (V^{12}+V^{34}) + e^{2} (V^{13}+V^{42}) - e^{3} (V^{14}+V^{23}), 
\end{align*}
where $\{e^{i}, V^{j} \}$ is a dual coframe of $\{e_{i}, V_{j} \}$. 
By the definition of the cross product 
in Remark \ref{def of cross product}, we have 
\begin{align*}
\left( e_{i} \times V_{j} \right) = 
\left( 
\begin{array}{cccc}
V_{2} & -V_{1} & V_{4} & -V_{3} \\
V_{3} & -V_{4} & -V_{1} & V_{2} \\
-V_{4}& -V_{3} & V_{2} & V_{1} 
\end{array} 
\right). 
\end{align*}
\end{lem}

\begin{lem} \label{diff_cross} \cite{K deform} 
For any 
$X, u, v \in \mathfrak{X}(A), \eta \in C^{\infty} (A, \nu)$, we have 
\begin{align*}
\nabla_{X}^{\perp_{A}} (u \times \eta) 
=&
(\nabla_{X}^{\top_{A}} u) \times \eta + u \times (\nabla_{X}^{\perp_{A}} \eta) - (\chi (X, u, \eta))^{\perp_{A}}, 
\end{align*}
where 
$\chi (X, u, \eta) = X \times (u \times \eta) + g(X, u) \eta$ 
and $\top_{A}: TY \rightarrow TA$ and $\perp_{A}: TY \rightarrow \nu$ are projections. 
\end{lem}

We can compute $\nabla_{e_{i}}^{\perp_{A}} V_{j}$ 
from $\nabla_{e_{i}}^{\top_{A}} e_{j}$ and $\nabla_{e_{i}}^{\perp_{A}} V_{1}$ by Lemma \ref{diff_cross} 
and obtain the following. 
The proof is straightforward and we omit it.

\begin{lem} \label{conn of Vj}
Denote 
$\nabla_{e_{i}}^{\top_{A}} e_{j} = \sum_{k=1}^{3} \Gamma_{i j}^{k} e_{k}$ and 
$\nabla_{e_{i}}^{\perp_{A}} V_{1} = \sum_{j=2}^{4} K_{i j} V_{j}$. Then we have 
for $i=1,2,3$
\begin{align*}
\nabla_{e_{i}}^{\perp_{A}} V_{2} 
&=
-K_{i 2} V_{1} + (\Gamma_{i 1}^{2} -K_{i 4} + \delta_{i 3}) V_{3} + (- \Gamma_{i 1}^{3} + K_{i 3} + \delta_{i 2}) V_{4}, \\
\nabla_{e_{i}}^{\perp_{A}} V_{3} 
&=
-K_{i 3} V_{1} + (\Gamma_{i 2}^{1} + K_{i 4} - \delta_{i 3}) V_{2} + (- \Gamma_{i 2}^{3} - K_{i 2} - \delta_{i 1}) V_{4}, \\
\nabla_{e_{i}}^{\perp_{A}} V_{4} 
&=
-K_{i 4} V_{1} + (- \Gamma_{i 3}^{1} -K_{i 3} - \delta_{i 2}) V_{3} + (- \Gamma_{i 3}^{2} + K_{i 2} + \delta_{i 1}) V_{3}. 
\end{align*}
\end{lem}

By the definition of the Levi-Civita connection, we have the following. 
The proof is also straightforward and we omit it. 
\begin{lem} \label{LC conn on A}
Suppose that 
$A$ is a Lie group $G$ and 
$\{ e_{i} \}_{i=1,2,3}$ are left invariant vector fields. 
Denoting $[e_{i}, e_{j}] = \sum_{i=1}^{3} c_{ij}^{k} e_{k}$ $(c_{ij}^{k} \in \mathbb{R})$, 
we have 
\begin{align*}
\nabla_{e_{i}}^{\top_{A}} e_{j} = 
\frac{1}{2} \sum_{k=1}^{3} \left(c_{i j}^{k} - c_{i k}^{j} - c_{j k}^{i} \right) e_{k}.
\end{align*}
\end{lem}


\subsection{Computations on ${\rm SU}(2)$}

For the convenience of the computation, 
we summarize formulas on ${\rm SU}(2)$. 
Define the basis 
$\{ E_{1}, E_{2}, E_{3} \}$
of $\mathfrak{su}(2)$ as (\ref{basis of su2}). 
\begin{lem} \label{Fourier_S3}
Let 
$V_{n}$ be a $\mathbb{C}$-vector space 
of all complex homogeneous polynomials 
with two variables $z_{1}, z_{2}$ of degree $n$, where $n \geq 0$, and 
define the representation 
$\rho_{n} : {\rm SU}(2) \rightarrow {\rm GL}(V_{n})$ as 
\begin{align*} 
\left(\rho_{n} 
\left( 
\begin{array}{cc}
a & -\overline{b} \\
b &   \overline{a} \\
\end{array} 
\right) f
\right) (z_{1}, z_{2}) 
=
f
\left( 
(z_{1}, z_{2}) 
\left(
\begin{array}{cc}
a & -\overline{b} \\
b &   \overline{a} \\
\end{array} 
\right) 
\right).  
\end{align*} 
Define the Hermitian inner product $\langle \ , \ \rangle$ 
of $V_{n}$ such that 
\begin{align*} 
\left \{
v^{(n)}_{k}
=
\frac{1}{\sqrt{k ! (n - k)!}} z_{1}^{n-k} z_{2}^{k}
\right \}_{0 \leq k \leq n}  
\end{align*} 
is a unitary basis of $V_{n}$. 
Denoting by $\widehat{{\rm SU}(2)}$  
the set of all equivalence classes of finite dimensional irreducible representations of ${\rm SU}(2)$, 
we know that 
$\widehat{{\rm SU}(2)} = \{ (V_{n} ,\rho_{n}) ; n \geq 0 \}$.  
Then 
every $\mathbb{C}$-valued continuous function 
on ${\rm SU}(2)$ is 
uniformly approximated by 
the $\mathbb{C}$-linear combination of the following functions: 
\begin{align*}
\left \{ \langle \rho_{n} (\cdot) v^{(n)}_{i}, v^{(n)}_{j} \rangle ; 
n \geq 0, 0 \leq i, j \leq n
\right \}, 
\end{align*} 
which are mutually orthogonal with respect to the $L_{2}$ inner product. 
\end{lem}

By a direct computation, we see the following. 

\begin{lem} \label{calc_S3}
Identify $X \in \mathfrak{su}(2)$ with 
the left invariant differential operator on ${\rm SU}(2)$. 
For $u = \sum_{l = 0}^{n} C_{l} v^{(n)}_{l} \in V_{n}$, 
set  
\begin{align*}
u^{*} = \sum_{l = 0}^{n} (-1)^{n - l} \overline{C}_{n - l} v^{(n)}_{l} \in V_{n}. 
\end{align*}
Then 
for any $n \geq 0, 0 \leq k, l \leq n,  u, v \in V_{n}, X \in \mathfrak{su}(2)$, we have 
\begin{align*}
X \langle \rho_{n} (\cdot) v, u \rangle
&=
\langle \rho_{n} (\cdot)  d \rho_{n}  (X) v, u \rangle, \\ 
(d \rho_{n} (X) v) (z_{1}, z_{2}) 
&= 
\left( \frac{\partial v}{\partial z_{1}}, \frac{\partial v}{\partial z_{2}} \right) {}^t\! X
\left(
\begin{array}{l}
z_{1} \\
z_{2}
\end{array}
\right), \\
\overline{ \langle \rho_{n} (\cdot) v^{(n)}_{k}, u \rangle} 
&=
(-1)^{k}
\langle \rho_{n} (\cdot) v^{(n)}_{n - k}, u^{*} \rangle,  
\end{align*}
\begin{align*}
(-i E_{1} + E_{2}) 
\langle \rho_{n} (\cdot) v^{(n)}_{k}, u \rangle
&=
\left\{
\begin{array}{ll}
2i \sqrt{(k + 1)(n - k)} \langle \rho_{n} (\cdot) v^{(n)}_{k + 1}, u \rangle, & (k < n) \\
0,                                                                                                      & (k = n)
\end{array}
\right.\\
(i E_{1} + E_{2}) 
\langle \rho_{n} (\cdot) v^{(n)}_{k}, u \rangle
&=
\left\{
\begin{array}{ll}
2i \sqrt{k (n - k + 1)} \langle \rho_{n} (\cdot) v^{(n)}_{k - 1}, u \rangle, & (k > 0) \\
0,                                                                                                      & (k = 0)
\end{array}
\right.\\
i E_{3} \langle \rho_{n} (\cdot) v^{(n)}_{k}, u \rangle
&= (-n + 2k) \langle \rho_{n} (\cdot) v^{(n)}_{k}, u \rangle. \\ 
\end{align*}
\end{lem}

\begin{lem} \label{gen_D}
Suppose that $\{ e_{1}, e_{2}, e_{3} \} = \{ p E_{1}, p E_{2}, q E_{3} \}$, 
where $0 \neq p, q \in \mathbb{R}$,  
is an oriented orthonormal basis of $\mathfrak{su}(2)$ 
for some metric and orientation. 
Define the differential operator 
$D_{\lambda, \mu}: C^{\infty}({\rm SU}(2), \mathbb{R}^{4}) \rightarrow C^{\infty}({\rm SU}(2), \mathbb{R}^{4})$ by 
\begin{align} \label{D lambda nu}
D_{\lambda, \mu} 
\left(
\begin{array}{c}
\psi_{1} \\
\psi_{2} \\
\psi_{3} \\
\psi_{4}
\end{array} 
\right) 
= 
\left \{
\left(
\begin{array}{cccc}
0       & -e_{1} & -e_{2} & e_{3} \\
e_{1}   & 0       & e_{3}  & e_{2} \\
e_{2}   & -e_{3} & 0      & -e_{1} \\
-e_{3} & -e_{2} & e_{1}  & 0 \\
\end{array} 
\right)
+
\left(
\begin{array}{cccc}
\lambda  &        &         &            \\
             & \mu &         &            \\
             &       & \mu   &            \\
             &       &         & \lambda \\
\end{array} 
\right)
\right \}
\left(
\begin{array}{c}
\psi_{1} \\
\psi_{2} \\
\psi_{3} \\
\psi_{4}
\end{array} 
\right), 
\end{align}
for $\lambda, \mu \in \mathbb{R}$. 
Setting $\Psi_{1} = \psi_{1} + i \psi_{4}, \Psi_{1} = \psi_{2} - i \psi_{3}$, 
$D_{\lambda, \mu}$ is described as 
\begin{align*}
D_{\lambda, \mu} 
\left(
\begin{array}{c}
\Psi_{1} \\
\Psi_{2} 
\end{array} 
\right)
=
\left \{
\left(
\begin{array}{cccc}
-i e_{3}       & -e_{1}-i e_{2}  \\
e_{1}-i e_{2} & i e_{3}            \\
\end{array} 
\right)
+
\left(
\begin{array}{cccc}
\lambda  &       \\
             & \mu 
\end{array} 
\right)
\right \}
\left(
\begin{array}{c}
\Psi_{1} \\
\Psi_{2} 
\end{array} 
\right).
\end{align*}
Set $\psi = {}^t\!(\psi_{1}, \psi_{2}, \psi_{3}, \psi_{4})$. 
Then $D_{\lambda, \mu} \psi = \alpha \psi$  for $\alpha \in \mathbb{R}$ is equivalent to 
\begin{align}
(-i e_{3} + \lambda - \alpha) \Psi_{1} - (e_{1} + i e_{2}) \Psi_{2} = 0, \label{gen_D1}\\
(e_{1} -i e_{2}) \Psi_{1} + (i e_{3} + \nu - \alpha) \Psi_{2} = 0 \label{gen_D2}.
\end{align}
These equations imply that 
$\Gamma_{p, q, \lambda, \mu, \alpha} \Psi_{2} = 0$, where 
$\Gamma_{p, q, \lambda, \mu, \alpha}$ is defined by 
\begin{align}
\Gamma_{p, q, \lambda, \mu, \alpha}
= 
\Delta_{+} + \left ( \mu - \lambda + 2q - \frac{2 p^{2}}{q} \right) i e_{3} 
+ (-2q+ \lambda -\alpha)( - \mu + \alpha), \label{gen_D3} 
\end{align}
where $\Delta_{+} = - \sum_{i = 1}^{3} e_{i}^{2}$ is a Laplacian on ${\rm SU}(2)$. 
Especially, for any $n \geq 0, 0 \leq k \leq n, u \in V_{n}$, we have 
\begin{align}
\Delta_{+} \langle \rho_{n} (\cdot) v^{(n)}_{k}, u \rangle
=&
\left \{
(-p^{2} + q^{2})(n - 2k)^{2} 
+
p^{2} (n^{2} + 2n)
\right \}
\langle \rho_{n} (\cdot) v^{(n)}_{k}, u \rangle, \label{gen_D4} \\
\Gamma_{p, q, \lambda, \mu, \alpha}
\langle \rho_{n} (\cdot) v^{(n)}_{k}, u \rangle  
=&
\left \{
(-p^{2} + q^{2}) (n - 2k)^{2} \right. \nonumber \\
&+ 
p^{2} (n^{2} + 2n) - (q(-\mu + \lambda) + 2(p^{2} - q^{2}) ) (n-2k)  \nonumber \\
&
\left.
+ 
(-2q+ \lambda -\alpha)( - \mu + \alpha)
\right \}
\langle \rho_{n} (\cdot) v^{(n)}_{k}, u \rangle. \label{gen_D5}
\end{align}
\end{lem}

\begin{rem}
In the case of ${\rm SU}(2)/ \Gamma$ for some finite subgroup $\Gamma$, 
we may consider the $\Gamma$ equivariant solutions of (\ref{gen_D1}) and (\ref{gen_D2}). 
\end{rem}

\begin{proof}
It is straightforward to derive (\ref{gen_D1}) and (\ref{gen_D2}). 
Since 
$
[e_{1}, e_{2}] = \frac{2p^{2}}{q} e_{3}, 
[e_{2}, e_{3}] =  2q e_{1}, 
[e_{3}, e_{1}] = 2q e_{2}, 
$
we have 
$(e_{1}-i e_{2}) i e_{3} = (i e_{3} + 2q) (e_{1} - i e_{2})$. 
Applying $(e_{1} -i e_{2})$ to (\ref{gen_D1}), we obtain 
\begin{align}
(- i e_{3} -2q + \lambda - \alpha) (e_{1} - i e_{2}) \Psi_{1}  
+ 
\left( - e_{1}^{2} - e_{2}^{2} - \frac{2p^{2}}{q} i e_{3} \right) \Psi_{2} = 0. \label{gen_D6}
\end{align}
Eliminating $\Psi_{1}$ from (\ref{gen_D6}) by (\ref{gen_D2}) gives (\ref{gen_D3}). 
From Lemma \ref{calc_S3}, we obtain (\ref{gen_D4}) and (\ref{gen_D5}). 
\end{proof}



\subsection{The case $L_{1}$}

Let ${\rm SU}(2) = {\rm Sp}(1)$ act on $S^{7}$ as (\ref{right sp1 action}). 
Then $L_{1}$ is the ${\rm SU}(2)$-orbit through $p_{0} = {}^t\!(1, 0, 0, 0)$. 
 Identifying 
${\rm SU}(2) \ni 
\left(
\begin{array}{cc}
a& - \overline{b} \\
b& \overline{a}   \\
\end{array} 
\right) 
\mapsto
a-\overline{b} j \in {\rm Sp}(1)$, 
the vector fields $E_{i}^{*}$ generated by 
$E_{i} \in \mathfrak{su}(2)$, where $i=1,2,3$, in (\ref{basis of su2}) 
are described as 
\begin{align*}
E_{1}^{*} = {}^t\!(1, 0, 0, 0) = -\xi_{2}, \qquad
E_{2}^{*} = {}^t\!(0, i, 0, 0) = -\xi_{3}, \qquad
E_{3}^{*} = {}^t\!(i, 0, 0, 0) = -\xi_{1},
\end{align*}
at $p_{0}$, 
which induces the orthonormal basis 
$\{ e_{1}, e_{2}, e_{3} \} = 5/3 \{E_{1}, E_{2}, -E_{3} \}$ of $\mathfrak{su}(2)$.

Set $v_{1} = \frac{\sqrt{5}}{3} {}^t\!(0, 0, 1, 0) \in \nu_{p_{0}}$, 
which is horizontal and $|v_{1}|_{\tilde{g}}= 1$. 
Denote $X_{0} = {}^t\!(0, 0, 1, 0)$, which is horizontal at $p_{0}$ 
and $X_{i} = \Phi_{i} (X_{0})$ for $i=1, 2, 3$. 
By the definition of $\tilde{\varphi}$ in Proposition \ref{can var}, 
the vectors $v_{2} = e_{1} \times v_{1}, v_{3} = e_{2} \times v_{1}, v_{4} = -e_{3} \times v_{1}$ are
described as $\{ v_{1}, v_{2}, v_{3}, v_{4} \} = \frac{5}{3} \{ X_{0}, X_{2}, X_{3}, X_{1} \}$. 
Define the vector field $V_{i}$ on $L_{1}$ by 
$(V_{i})_{g \cdot p_{0}} = g_{*} v_{i}$, where $g \in {\rm SU}(2)$, we obtain the following 
by Lemma \ref{relation of connection 3-sasaki} and Lemma \ref{LC conn on A}.  
\begin{align*}
\left(
\begin{array}{c}
\tilde{\nabla}_{e_{1}} V_{1} \\
\tilde{\nabla}_{e_{2}} V_{1} \\
\tilde{\nabla}_{e_{3}} V_{1}
\end{array} 
\right) 
=
\frac{1}{3}
\left(
\begin{array}{c}
V_{2} \\
V_{3} \\
-V_{4}
\end{array} 
\right), \qquad
(\tilde{\nabla}^{\top_{L_{1}}}_{e_{i}} e_{j}) = 
\frac{5}{3}
\left(
\begin{array}{ccc}
0    & -e_{3} & e_{2} \\
e_{3}& 0       & -e_{1} \\
-e_{2}& e_{1} & 0
\end{array} 
\right). 
\end{align*}
This computation and Lemma \ref{conn of Vj} give the following. 
\begin{align*}
(\tilde{\nabla}_{e_{i}} V_{j}) = 
\frac{1}{3}
\left(
\begin{array}{cccc}
V_{2} & -V_{1} & V_{4} & -V_{3}\\
V_{3} & -V_{4}& -V_{1}& V_{2}\\
-V_{4}& -V_{3}& V_{2}& V_{1}
\end{array} 
\right). 
\end{align*}
Then by the trivialization of $\nu$ via $\{ V_{1}, V_{2}, V_{3}, V_{4} \}$, 
we have 
$D= D_{-1,-1}$, where $D_{\lambda, \mu}$ is defined in (\ref{D lambda nu}). 
Using the notations of Lemma \ref{gen_D}, 
we see that $\Psi_{2}$ is constant, and hence $\Phi_{1}$ is constant. 
Thus 
we obtain $\dim_{\mathbb{R}} \{ \psi \in C^{\infty}(L_{1}, \nu); D \psi = - \psi \} = 4$. 

Since 
$\dim_{\mathbb{R}} {\rm Sp}(1) {\rm Sp}(2)/ {\rm Sp}(1) ({\rm Sp}(1) \times {\rm Sp}(1))=4$, 
${\rm Sp}(1) {\rm Sp}(2)$ induces 
$4$-dimensional  associative deformations of $L_{1}$ and we obtain the following.

\begin{prop} 
The associative deformations of $L_{1}$ are trivial. 
Its deformation space is ${\rm Sp}(1) {\rm Sp}(2)/ {\rm Sp}(1) ({\rm Sp}(1) \times {\rm Sp}(1))
= \mathbb{H}P^{1} = S^{4}$. 
The associative deformations of $L_{1}$ are 
the deformations of fibers of $\pi : S^{7} \rightarrow S^{4}$ parametrized by 
the base space $S^{4}$. 
\end{prop}


\subsection{The case $L_{2}$}

Let ${\rm SU}(2)$ act on $S^{7}$ by (\ref{diag SU2}). 
Then $L_{2}$ is the ${\rm SU}(2)$-orbit through $p_{0} = {}^t\!(1, 0, 0, 0)$. 
By (\ref{lie alg of diag SU2}), 
the vector fields $E_{i}^{*}$ generated by 
$E_{i} \in \mathfrak{su}(2)$ for $i=1,2,3$ in (\ref{basis of su2}) 
are described as 
\begin{align*}
E_{1}^{*} = {}^t\!(0, 0, -1, 0), \qquad
E_{2}^{*} = {}^t\!(0, 0, -i, 0), \qquad
E_{3}^{*} = {}^t\!(-i, 0, 0, 0) = \xi_{1},
\end{align*}
and satisfy  $\tilde{\varphi}(E_{1}^{*}, E_{2}^{*}, E_{3}^{*}) = -27/25 < 0$
at $p_{0}$. 
Then we obtain the induced oriented 
orthonormal basis 
$\{ e_{1}, e_{2}, e_{3} \} = \{ \frac{\sqrt{5}}{3} E_{1}, \frac{\sqrt{5}}{3} E_{2}, - \frac{5}{3} E_{3} \}$ 
of $\mathfrak{su}(2)$.

Set $v_{1} = \frac{5}{3} {}^t\!(0, 1, 0, 0) = - \frac{5}{3} \xi_{2} \in \nu_{p_{0}}$, 
which  satisfies $|v_{1}|_{\tilde{g}}= 1$. 
Denote $X_{0} = {}^t\!(0, 0, 1, 0)$, which is horizontal at $p_{0}$ 
and $X_{i} = \Phi_{i} (X_{0})$ for $i=1, 2, 3$. 
Since 
$\{ e_{1}, e_{2}, e_{3} \} = \{ -\frac{\sqrt{5}}{3} X_{0}, -\frac{\sqrt{5}}{3} X_{1}, \xi_{1} \}$, 
vectors $v_{2} = e_{1} \times v_{1}, v_{3} = e_{2} \times v_{1}, v_{4} = -e_{3} \times v_{1}$ are
described as $\{ v_{1}, v_{2}, v_{3}, v_{4} \} = \{ -\frac{5}{3} \xi_{2}, \frac{\sqrt{5}}{3} X_{2}, - \frac{\sqrt{5}}{3} X_{3}, -\frac{5}{3} \xi_{3} \}$. 
Define the vector field $V_{i}$ on $L_{2}$ by 
$(V_{i})_{g \cdot p_{0}} = g_{*} v_{i}$ where $g \in {\rm SU}(2)$. 
As in the case $L_{1}$, we obtain
\begin{align*}
(\tilde{\nabla}_{e_{i}} V_{j}) = 
\frac{1}{3}
\left(
\begin{array}{cccc}
-V_{2} & V_{1} & -V_{4} & V_{3}\\
-V_{3} & V_{4}&  V_{1}& -V_{2}\\
-5V_{4}& -V_{3}& V_{2}& 5 V_{1}
\end{array} 
\right). 
\end{align*}
Then by the trivialization of $\nu$ via $\{ V_{1}, V_{2}, V_{3}, V_{4} \}$, 
we have 
$D= D_{-1,-1/3}$, where $D_{\lambda, \mu}$ is defined in (\ref{D lambda nu}). 
Setting $(p, q, \lambda, \mu, \alpha) = (\frac{\sqrt{5}}{3}, - \frac{5}{3}, -1, \frac{1}{3}, -1)$ in (\ref{gen_D5}), 
we see that 
\begin{align*}
\Psi_{2} = \langle \rho_{2} (\cdot) v^{(2)}_{1}, u \rangle
\end{align*}
for $u \in V_{2}$. 
Since $\ker (e_{1} -i e_{2}) \cap \ker (i e_{3}) = \mathbb{C}$, 
(\ref{gen_D1}) and (\ref{gen_D2}) imply that 
\begin{align*}
\Psi_{1} = - \frac{\sqrt{10}}{5} \langle \rho_{2} (\cdot) v^{(2)}_{2}, u \rangle + C
\end{align*}
for $C \in \mathbb{C}$. 
Thus 
we obtain $\dim_{\mathbb{R}} \{ \psi \in C^{\infty}(L_{2}, \nu); D \psi = - \psi \} = 8$. 

Since 
$\dim_{\mathbb{R}} {\rm Sp}(1) {\rm Sp}(2)/  {\rm U}(1) {\rm U}(2) = 8$, 
${\rm Sp}(1) {\rm Sp}(2)$ induces 
$8$-dimensional associative deformations of $L_{2}$ and we obtain the following.

\begin{prop} 
The associative deformations of $L_{2}$ are trivial. 
Its deformation space is $ {\rm Sp}(1) {\rm Sp}(2)/ {\rm U}(1) {\rm U}(2)$.
\end{prop}


\subsection{The case $A_{1}$}

Let $T^{3}$ act on $S^{7}$ by (\ref{T3 action}).  
Then $A_{1}$ is the $T^{3}$-orbit through $p_{0} = \frac{1}{2} {}^t\!(1, 1, 1, i)$. 
By (\ref{lie alg of T3}), 
the vector fields $F_{i}^{*}$ generated by 
$F_{i}$ for $i=1,2,3$ in (\ref{basis of t3}) are described as 
\begin{align*}
F_{1}^{*} = \frac{1}{2} {}^t\!(i, i, i, -1) = - \xi_{1}, \qquad
F_{2}^{*} = \frac{1}{2} {}^t\!(i, -i, 0, 0), \qquad
F_{3}^{*} = \frac{1}{2} {}^t\!(0, 0, i, 1),
\end{align*}
and satisfy  $\tilde{\varphi}(F_{1}^{*}, F_{2}^{*}, F_{3}^{*}) = -81/250 < 0$ at $p_{0}$. 
Then we obtain the induced oriented 
orthonormal basis 
$\{ e_{1}, e_{2}, e_{3} \} = \{ \frac{5}{3} F_{1}, \frac{5 \sqrt{6}}{9} F_{2}, - \frac{5 \sqrt{6}}{9} F_{3} \}$ 
of $\mathfrak{t}^{3}$.

Set $v_{1} = \frac{\sqrt{5}}{6} {}^t\!(-1, -1, 1, i)$, 
which is horizontal at $p_{0}$ and $|v_{1}|_{\tilde{g}}= 1$. 
Denote $X_{0} = \frac{1}{2} {}^t\!(-1, -1, 1, i)$, which is horizontal at $p_{0}$ 
and $X_{i} = \Phi_{i} (X_{0})$ for $i=1, 2, 3$. 
Since 
\begin{align*}
e_{1} = - \frac{5}{3} \xi_{1}, \qquad
e_{2} = \frac{5 \sqrt{6}}{18} (\xi_{3}+X_{3}), \qquad
e_{3} = \frac{5 \sqrt{6}}{18} (\xi_{2}-X_{2}), 
\end{align*}
vectors $v_{2} = e_{1} \times v_{1}, v_{3} = e_{2} \times v_{1}, v_{4} = -e_{3} \times v_{1}$ are
described as 
\begin{align*}
\{ v_{1}, v_{2}, v_{3}, v_{4} \} = 
\left \{ \frac{\sqrt{5}}{3} X_{0}, \frac{\sqrt{5}}{3} X_{1}, 
\frac{\sqrt{30}}{18} (-X_{3}+ 5 \xi_{3}), \frac{\sqrt{30}}{18} (X_{2}+ 5 \xi_{2}) 
\right \}. 
\end{align*}
Define the vector field $V_{i}$ on $T^{3}$ by 
$(V_{i})_{g \cdot p_{0}} = g_{*} v_{i}$, where $g \in T^{3}$.
As in the case $L_{1}$, we obtain 
\begin{align*}
(\tilde{\nabla}_{e_{i}}^{\perp_{A_{1}}} V_{j}) = 
\frac{1}{9}
\left(
\begin{array}{cccc}
3 V_{2} & -3 V_{1} & -12 V_{4} & 12 V_{3}\\
-2V_{3} & 7 V_{4}&  2 V_{1}& -7 V_{2}\\
2 V_{4}& 7 V_{3}& -7 V_{2}& -2 V_{1}
\end{array} 
\right). 
\end{align*}
Then by the trivialization of $\nu$ via $\{ V_{1}, V_{2}, V_{3}, V_{4} \}$, 
we have 
\begin{align*}
D = 
\left(
\begin{array}{cccc}
0       & -e_{1} & -e_{2} & e_{3} \\
e_{1}   & 0       & e_{3}  & e_{2} \\
e_{2}   & -e_{3} & 0      & -e_{1} \\
-e_{3} & -e_{2} & e_{1}  & 0 \\
\end{array} 
\right)
+
\frac{1}{9}
\left(
\begin{array}{cccc}
1           &        &         &            \\
             & 11   &         &            \\
             &       & 21     &            \\
             &       &         & 21 \\
\end{array} 
\right). 
\end{align*}
Suppose $D \psi = - \psi$, where $\psi = {}^t\!(\psi_{1}, \psi_{2}, \psi_{3}, \psi_{4})$ 
and $\psi_{i} \in C^{\infty}(T^{3})$. 
Eliminating $\psi_{2}$ by $\psi_{2} = -\frac{9}{20} (e_{1} (\psi_{1}) + e_{3} (\psi_{3}) + e_{2} (\psi_{4}))$, 
we obtain 
\begin{align}
\left( \frac{10}{9} + \frac{9}{20} e_{1}^{2} \right) \psi_{1} + 
\left( \frac{9}{20} e_{1} e_{3} - e_{2} \right) \psi_{3} + 
\left( \frac{9}{20} e_{1} e_{2} + e_{3} \right) \psi_{4} = 0, \label{T3 deform1} \\ 
\left( \frac{9}{20} e_{1} e_{3} + e_{2} \right) \psi_{1} + 
\left( \frac{10}{3} + \frac{9}{20} e_{3}^{2} \right) \psi_{3} + 
\left( \frac{9}{20} e_{2} e_{3} - e_{1} \right) \psi_{4} = 0, \label{T3 deform2} \\
\left( \frac{9}{20} e_{1} e_{2} - e_{3} \right) \psi_{1} + 
\left( \frac{9}{20} e_{2} e_{3} + e_{1} \right) \psi_{3} + 
\left( \frac{10}{3} + \frac{9}{20} e_{2}^{2} \right) \psi_{4} = 0. \label{T3 deform3}
\end{align}

Define the smooth function $f_{\gamma} \in C^{\infty}(T^{3}, \mathbb{C})$ 
for  $\gamma = (\gamma_{1}, \gamma_{2}, \gamma_{3}) \in \mathbb{Z}^{3}$ 
on $T^{3} \cong  (\mathbb{R}/ 2 \pi \mathbb{Z})^{3}$ 
by 
$
f_{\gamma}(\theta_{1}, \theta_{2}, \theta_{3}) = \exp (i \sum_{j = 1}^{3} \gamma_{j} \theta_{j}).  
$
Identifying $e_{i} \in \mathfrak{t}^{3}$ with the left invariant differential operator on $T^{3}$, 
we have 
\begin{align*}
e_{1}(f_{\gamma}) = \frac{5}{3} \gamma_{1} i f_{\gamma}, \qquad
e_{2}(f_{\gamma}) = \frac{5 \sqrt{6}}{9} \gamma_{2} i f_{\gamma}, \qquad
e_{3}(f_{\gamma}) = -\frac{5 \sqrt{6}}{9} \gamma_{3} i f_{\gamma}.  
\end{align*}
By a Fourier series expansion, set 
\begin{align*}
\psi_{1} = \sum_{\gamma \in \mathbb{Z}^{3}} C_{\gamma} f_{\gamma}, \qquad
\psi_{2} = \sum_{\gamma \in \mathbb{Z}^{3}} D_{\gamma} f_{\gamma}, \qquad
\psi_{3} = \sum_{\gamma \in \mathbb{Z}^{3}} E_{\gamma} f_{\gamma}, 
\end{align*}
where $C_{\gamma}, D_{\gamma}, E_{\gamma} \in \mathbb{C}$.
Then (\ref{T3 deform1}), (\ref{T3 deform2}), and (\ref{T3 deform3}) are equivalent to 
$M_{\gamma} {}^t\!(C_{\gamma}, D_{\gamma}, E_{\gamma}) = 0$, where 
\begin{align*}
M_{\gamma}
=
\left(
\begin{array}{ccc}
8-9 \gamma_{1}^{2}                                        & 3 \sqrt{6} \gamma_{1} \gamma_{3} - 4 \sqrt{6} \gamma_{2} i  
& -3 \sqrt{6} \gamma_{1} \gamma_{2} - 4 \sqrt{6} \gamma_{3} i  \\
3 \sqrt{6} \gamma_{1} \gamma_{3} + 4 \sqrt{6} \gamma_{2} i  & -6 \gamma_{3}^{2} + 24       & 6 \gamma_{2} \gamma_{3} - 12 \gamma_{1} i\\
-3 \sqrt{6} \gamma_{1} \gamma_{2} + 4 \sqrt{6} \gamma_{3} i    & 6 \gamma_{2} \gamma_{3} + 12 \gamma_{1} i    & -6 \gamma_{2}^{2} + 24
\end{array} 
\right). 
\end{align*}
To obtain a nontrivial solution ${}^t\!(C_{\gamma}, D_{\gamma}, E_{\gamma}) \neq 0$, 
\begin{align*}
\det M_{\gamma} = 16 \left \{ (9 \gamma_{1}^{2} + 6 \gamma_{2}^{2} + 6 \gamma_{3}^{2} -22)^{2} 
+ 4(12(\gamma_{2}^{2} + \gamma_{3}^{2}) -49) \right \}
\end{align*}
must vanish. We see that $\det M_{\gamma} = 0$ if and only if 
\begin{align} \label{solution gamma}
(\gamma_{1}, \gamma_{2}, \gamma_{3}) = \pm (2, 0, 0), \pm (0, 2, 0), \pm (0, 0, 2), \pm (0, 1, 1), \pm (0, 1, -1). 
\end{align}
For each $\gamma$ in (\ref{solution gamma}), we can check $\dim \ker M_{\gamma} = 1$. 
Moreover, we have 
$C_{\gamma} = \overline{C_{- \gamma}}, D_{\gamma} = \overline{D_{- \gamma}}$, 
and $E_{\gamma} = \overline{E_{- \gamma}}$ 
so that every $\psi_{j}$ is $\mathbb{R}$-valued. 
Hence 
we obtain $\dim_{\mathbb{R}} \{ \psi \in C^{\infty}(A_{1}, \nu); D \psi = - \psi \} = 10$.

Since 
$\dim_{\mathbb{R}} {\rm Sp}(1) {\rm Sp}(2)/ T^{3} = 10$, 
${\rm Sp}(1) {\rm Sp}(2)$ induces 
$10$-dimensional 
associative deformations of $A_{1}$ and we obtain the following.

\begin{prop} 
The associative deformations of $A_{1}$ are trivial. 
Its deformation space is ${\rm Sp}(1) {\rm Sp}(2)/ T^{3}$.
\end{prop}


\subsection{The case $A_{2}$}

Let ${\rm SU}(2)$ act on $S^{7}$ by (\ref{irr SU2}). 
Then $A_{2}$ is the ${\rm SU}(2)$-orbit through $p_{0} = {}^t\!(1, 0, 0, 0)$. 
By (\ref{vfs A2}), 
$\{ e_{1}, e_{2}, e_{3} \} = \{ \frac{\sqrt{15}}{9} E_{1}, \frac{\sqrt{15}}{9} E_{2}, - \frac{5}{9} E_{3} \}$ 
is the induced oriented 
orthonormal basis 
of $\mathfrak{su}(2)$, 
where 
$E_{i} \in \mathfrak{su}(2)$ for $i=1,2,3$ is defined in (\ref{basis of su2}).

Set $v_{1} = \frac{5}{3} {}^t\!(0, 1, 0, 0) = - \frac{5}{3} \xi_{2} \in \nu_{p_{0}}$, 
which satisfies $|v_{1}|_{\tilde{g}}= 1$. 
Denote $X_{0} = {}^t\!(0, 0, 0, 1)$, which is horizontal at $p_{0}$ 
and $X_{i} = \Phi_{i} (X_{0})$ for $i=1, 2, 3$. 
Since 
\begin{align*}
e_{1} = - \frac{\sqrt{5}}{3} X_{0}, \qquad
e_{2} =  \frac{\sqrt{5}}{3} X_{1}, \qquad
e_{3} = - \frac{5}{3} \xi_{1}, 
\end{align*}
vectors $v_{2} = e_{1} \times v_{1}, v_{3} = e_{2} \times v_{1}, v_{4} = -e_{3} \times v_{1}$ are
described as 
\begin{align*}
\{ v_{1}, v_{2}, v_{3}, v_{4} \} = 
\left \{ 
-\frac{5}{3} \xi_{2}, 
\frac{\sqrt{5}}{3} X_{2}, \frac{\sqrt{5}}{3} X_{3}, 
\frac{5}{3} \xi_{3}
\right \}. 
\end{align*}
Define the vector field $V_{i}$ in the neighborhood of $p_{0}$ of $A_{2}$ by 
$(V_{i})_{g \cdot p_{0}} = g_{*} v_{i}$, where $g \in {\rm SU}(2)$.
As in the case $L_{1}$, we obtain 
\begin{align*}
(\tilde{\nabla}_{e_{i}}^{\perp_{A_{2}}} V_{j}) = 
\frac{1}{9}
\left(
\begin{array}{cccc}
-3 V_{2} & 3 V_{1} & -3 V_{4} & 3 V_{3}\\
-3 V_{3} & 3 V_{4}&  3 V_{1}   &   -3 V_{2}\\
-15 V_{4}& 17 V_{3}& -17 V_{2}& 15 V_{1}
\end{array} 
\right). 
\end{align*}
Then by the local trivialization of $\nu$ via $\{ V_{1}, V_{2}, V_{3}, V_{4} \}$, 
we have 
$D= D_{-1, 23/9}$, where $D_{\lambda, \mu}$ is defined in (\ref{D lambda nu}). 
Setting $(p, q, \lambda, \mu, \alpha) = (\frac{\sqrt{15}}{9}, - \frac{5}{9}, -1, \frac{23}{9}, -1)$ in (\ref{gen_D5}), 
we see that 
\begin{align*}
\Psi_{2} = \langle \rho_{6} (\cdot) v^{(6)}_{5}, u \rangle
\end{align*}
for $u \in V_{2}$. 
Since $\ker (e_{1} -i e_{2}) \cap \ker (i e_{3}) = \mathbb{C}$, 
(\ref{gen_D1}) and (\ref{gen_D2}) imply that 
\begin{align*}
\Psi_{1} = - \frac{\sqrt{10}}{5} \langle \rho_{6} (\cdot) v^{(6)}_{6}, u \rangle + C
\end{align*}
for $C \in \mathbb{C}$. 
These solutions are $\mathbb{Z}_{3}$-equivariant, 
and hence
we obtain $\dim_{\mathbb{R}} \{ \psi \in C^{\infty}(A_{2}, \nu); D \psi = - \psi \} = 16$. 

Since 
$\dim_{\mathbb{R}} {\rm Sp}(1) {\rm Sp}(2)/ {\rm U}(1) {\rm SU}(2) = 9$, 
${\rm Sp}(1) {\rm Sp}(2)$ induces 
$9$-dimensional 
associative deformations of $A_{2}$. 
Thus $A_{2}$ can have at most 7-dimensional family of nontrivial
associative deformations. In fact, we obtain the following. 

\begin{prop} \label{nontrivial deform A2}
All associative deformations of $A_{2}$ are 
induced by the ${\rm Sp}(1) {\rm Sp}(2)$-action
and by the ${\rm PSp}(2, \mathbb{C})$-action on $\mathbb{C}P^{3}$ 
via the Hopf lift.
In other words, 
all the associative deformations of $A_{2}$ are given by the following. 

\begin{itemize}
\item
the ${\rm PSp}(2, \mathbb{C})$-action on $\mathbb{C}P^{3}$ 
via the Hopf lift, 
which corresponds to 
the deformation of $p_{1} (A_{2})$ as a horizontal holomorphic curve, 
where $p_{1}: S^{7} \rightarrow \mathbb{C}P^{3}$ is a projection, 
\item
the action generated by $j, k \in {\rm Sp}(1)$.
\end{itemize}
\end{prop}

Note that 
${\rm PSp}(2, \mathbb{C})$ acts on $\mathbb{C}P^{3}$ 
as the group of biholomorphic maps which preserve 
the horizontal distribution \cite{Bolton}, \cite{Ohnita_def}.

\begin{proof}
First description is an analogue of \cite{Ohnita_def}, \cite{K deform} 
and we omit the proof. 
The second description follows from the next lemma.
\end{proof}

\begin{lem} \label{deform pi A2}
The subgroup of ${\rm PSp}(2, \mathbb{C})$ 
which preserves $p_{1} (A_{2})$ is isomorphic to 
${\rm PSL}(2, \mathbb{C})$. 
Thus the deformation space of $p_{1} (A_{2})$ as a holomorphic curve 
is ${\rm PSp}(2, \mathbb{C})/ {\rm PSL}(2, \mathbb{C})$, 
which is 14-dimensional. 
\end{lem}

\begin{proof}
The inclusion ${\rm SU}(2) \hookrightarrow {\rm Sp}(2)$ of (\ref{irr SU2}) is 
canonically extended to 
$GL(2, \mathbb{C}) \hookrightarrow {\rm GL}(4, \mathbb{C})$:
\begin{align*}
(g_{ij}) 
\mapsto 
\left(
\begin{array}{cccc}
g_{11}^{3}                    & g_{12}^{3}                     & \sqrt{3} g_{11} g_{12}^{2} & \sqrt{3} g_{11}^{2} g_{12} \\
g_{21}^{3}                    & g_{22}^{3}                     & \sqrt{3} g_{21} g_{22}^{2} & \sqrt{3} g_{21}^{2} g_{22} \\
\sqrt{3} g_{11} g_{21}^{2} & \sqrt{3} g_{12} g_{22}^{2} & g_{22}(g_{11} g_{22} + 2 g_{12} g_{21}) & g_{21} (2 g_{11} g_{22} + g_{12} g_{21}) \\
\sqrt{3} g_{11}^{2} g_{21} & \sqrt{3} g_{12}^{2} g_{22} &  g_{12} (2 g_{11} g_{22} + g_{12} g_{21})&  g_{11}(g_{11} g_{22} + 2 g_{12} g_{21})  \\
\end{array} 
\right), 
\end{align*}
which is the group of biholomorphic maps which preserve $p_{1} (A_{2})$. 
We can check that 
$GL(2, \mathbb{C}) \cap {\rm Sp}(2, \mathbb{C}) = {\rm SL}(2, \mathbb{C})$, 
and hence we obtain the proof.
\end{proof}

\subsection{The case $A_{3}$}

Let ${\rm SU}(2)$ act on $S^{7}$ by (\ref{irr SU2}). 
Then $A_{3}$ is the ${\rm SU}(2)$-orbit through $p_{0} = {}^t\!(0, 0, 1, 0)$. 
By (\ref{vfs A3}), 
$\{ e_{1}, e_{2}, e_{3} \} = \{ \frac{5 \sqrt{19}}{57} E_{1}, \frac{5\sqrt{19}}{57} E_{2}, \frac{5}{3} E_{3} \}$ 
is the induced oriented 
orthonormal basis 
of $\mathfrak{su}(2)$, 
where 
$E_{i} \in \mathfrak{su}(2)$ for $i=1,2,3$ is defined in (\ref{basis of su2}).

Set $v_{1} = \frac{\sqrt{5}}{3} {}^t\!(1, 0, 0, 0) \in \nu_{p_{0}}$, 
which is horizontal at $p_{0}$ and $|v_{1}|_{\tilde{g}}= 1$. 
Denote $X_{0} = {}^t\!(1, 0, 0, 0)$, which is horizontal at $p_{0}$ 
and $X_{i} = \Phi_{i} (X_{0})$ for  $i=1,2,3$. 
Since 
\begin{align} \label{onb A3}
e_{1} = - \frac{5 \sqrt{19}}{57} (2 \xi_{2} + \sqrt{3} X_{2}), \qquad
e_{2} =  \frac{5 \sqrt{19}}{57} (-2 \xi_{3} + \sqrt{3} X_{3}), \qquad
e_{3} =  \frac{5}{3} \xi_{1}, 
\end{align}
vectors $v_{2} = e_{1} \times v_{1}, v_{3} = e_{2} \times v_{1}, v_{4} = -e_{3} \times v_{1}$ are
described as 
\begin{align*}
\{ v_{1}, v_{2}, v_{3}, v_{4} \} = 
\left \{ 
\frac{\sqrt{5}}{3} X_{0}, 
\frac{\sqrt{95}}{57} (-5 \sqrt{3} \xi_{2} + 2 X_{0}), 
\frac{\sqrt{95}}{57} (5 \sqrt{3} \xi_{3} + 2 X_{3}), 
\frac{\sqrt{5}}{3} X_{1}
\right \}. 
\end{align*}
Define the vector field $V_{i}$ on ${\rm SU}(2)$ by 
$(V_{i})_{g \cdot p_{0}} = g_{*} v_{i}$, where $g \in {\rm SU}(2)$.
As in the case $L_{1}$, we obtain 
\begin{align*}
(\tilde{\nabla}_{e_{i}}^{\perp_{A_{3}}} V_{j}) = 
\frac{1}{57}
\left(
\begin{array}{cccc}
-31 V_{2} & 31 V_{1}  & -31 V_{4} & 31 V_{3}\\
-31 V_{3} & 31 V_{4}  &  31 V_{1}   & -31 V_{2}\\
361 V_{4}& -119 V_{3}& 119 V_{2}& -361 V_{1}
\end{array} 
\right). 
\end{align*}
Then by the local trivialization of $\nu$ via $\{ V_{1}, V_{2}, V_{3}, V_{4} \}$, 
we have 
$D= D_{141/19, -1}$, where $D_{\lambda, \mu}$ is defined in (\ref{D lambda nu}). 
Setting $(p, q, \lambda, \mu, \alpha) = (\frac{5 \sqrt{19}}{57}, \frac{5}{3}, \frac{141}{19}, -1, -1)$ in (\ref{gen_D5}), 
we see that 
\begin{align*}
\Psi_{2} = \langle \rho_{6} (\cdot) v^{(6)}_{4}, u \rangle + C
\end{align*}
for $u \in V_{2}, C \in \mathbb{C}$. 
Since $\ker (i e_{3} - \frac{160}{19}) = \{ 0 \}$, 
(\ref{gen_D1}) and (\ref{gen_D2}) imply that 
\begin{align*}
\Psi_{1} = - \frac{\sqrt{190}}{10} \langle \rho_{6} (\cdot) v^{(6)}_{5}, u \rangle. 
\end{align*} 
Hence 
we obtain $\dim_{\mathbb{R}} \{ \psi \in C^{\infty}(A_{3}, \nu); D \psi = - \psi \} = 16$. 

Since 
$\dim_{\mathbb{R}} {\rm Sp}(1) {\rm Sp}(2)/ {\rm SU}(2) = 10$, 
${\rm Sp}(1) {\rm Sp}(2)$ induces 
$10$-dimensional 
associative deformations of $A_{3}$. 
Thus $A_{3}$ can have at most 6-dimensional family of nontrivial
associative deformations.

The associative deformation space of $A_{3}$ 
is explained by a one-to-one correspondence 
between null-torsion $I_{1}^{'}$-holomorphic curves 
and horizontal holomorphic curves in $\mathbb{C}P^{3}$ (\cite{Xu pseudo}).

Decompose $T \mathbb{C}P^{3} = \underline{\mathcal{H}} \oplus \underline{\mathcal{V}}$,  
where $\underline{\mathcal{V}}$ is a vector bundle tangent to the fibers of 
$p_{2}: \mathbb{C}P^{3} \rightarrow S^{4}$, 
and $\underline{\mathcal{H}}$ is its orthogonal complement bundle of $\underline{\mathcal{V}}$. 
Define a map $P: \underline{\mathcal{H}} - \{ 0 \} \rightarrow \mathbb{C}P^{3}$ 
by $P(v)= [\tilde{v}]$, 
where $\tilde{v} \in \mathcal{H} \subset T S^{7}$ 
is a horizontal lift of $v$ with respect to $p_{1}: S^{7} \rightarrow \mathbb{C}P^{3}$ 
and we identify $\tilde{v}$ with a vector in $\mathbb{C}^{4}$.

Let $pr_{\underline{\mathcal{H}}}: T \mathbb{C}P^{3} \rightarrow \underline{\mathcal{H}}$ 
be a canonical projection
and $\Sigma \subset \mathbb{C}P^{3}$ be a $I_{1}^{'}$-holomorphic curve with 
$pr_{\underline{\mathcal{H}}}|_{T \Sigma} \neq 0$. 
Then there exist a holomorphic line bundle $L \subset \underline{\mathcal{H}}|_{\Sigma}$ such that 
$pr_{\underline{\mathcal{H}}}(T \Sigma) \subset L$.
If $pr_{\underline{\mathcal{H}}}$ is nowhere vanishing on $\Sigma$, 
$L = pr_{\underline{\mathcal{H}}}(T \Sigma)$. 
Denote by $L^{\perp_{\underline{\mathcal{H}}}} \subset \underline{\mathcal{H}}|_{\Sigma}$ 
the orthonormal complement bundle of $L$ 
and set $\hat{\Sigma} = P (L^{\perp_{\underline{\mathcal{H}}}}- \{ 0 \})$. 

\begin{definition} \label{null-torsion def}
A non-vertical $I_{1}^{'}$-holomorphic curve $\Sigma$ is called 
{\bf null-torsion} if $\hat{\Sigma}$ is a horizontal holomorphic curve. 
\end{definition}

\begin{prop} \cite{Xu pseudo} \label{one-to-one curve}
There is a one-to-one correspondence 
between null-torsion $I_{1}^{'}$-holomorphic curves 
and horizontal holomorphic curves via $\Sigma \mapsto \hat{\Sigma}$. 
\end{prop}

Since $p_{1} (A_{3})$ is an image of $\mathbb{C}P^{1}$, 
it is a null-torsion (\cite{Xu pseudo}). 
We see the following. 

\begin{lem} \label{one-to-one A3 A2}
By Proposition \ref{one-to-one curve}, $p_{1} (A_{3})$ corresponds to $p_{1} (A_{2})$. 
\end{lem}

\begin{proof}
Since $pr_{\underline{\mathcal{H}}}$ is nowhere vanishing on $p_{1} (A_{3})$, 
$L = pr_{\underline{\mathcal{H}}}(T (p_{1} (A_{3})))$. 
By (\ref{onb A3}), 
$T_{p_{1} (p_{0})} (p_{1} (A_{3}))$ is a projection of 
the subspace of $T_{p_{0}} S^{7}$ spanned by 
$-2 \xi_{2} - \sqrt{3} X_{2}$ and $-2 \xi_{3} + \sqrt{3} X_{3}$. 
Thus 
the vector bundle $\tilde{L}^{\perp_{\underline{\mathcal{H}}}}$ over $A_{3}$ 
whose fiber at $g \cdot p_{0}$, where $g \in {\rm SU}(2)$, is spanned 
by $g_{*} X_{0}$ and $g_{*} X_{1}$ 
satisfies $(p_{1})_{*} (\tilde{L}^{\perp_{\underline{\mathcal{H}}}}) = L^{\perp_{\underline{\mathcal{H}}}}$, 
which implies that 
\begin{align*}
\widehat{p_{1} (A_{3})} 
&= [L^{\perp_{\underline{\mathcal{H}}}} - \{ 0 \} ] \\
&= \{ [g {}^t\!(1,0,0,0)] \in \mathbb{C}P^{3}; g \in {\rm SU}(2) \} 
= p_{1} (A_{2}).
\end{align*}
\end{proof}

\begin{rem}
We easily see that $\widehat{p_{1} (A_{1})} = p_{1} (A_{1})$, and hence 
$p_{1} (A_{1})$ is not null-torsion. 
\end{rem}

Since the deformation space of $p_{1} (A_{2})$ 
as a horizontal holomorphic curve is 14-dimensional by Proposition \ref{nontrivial deform A2}, 
we obtain the following result.

\begin{prop} \label{nontrivial deform A3}
All the associative deformations of $A_{3}$ are given by the following. 

\begin{itemize}
\item
the Hopf lift of null-torsion $I_{1}^{'}$-holomorphic curves, 
which correspond 
to horizontal holomorphic curves 
obtained by deforming $p_{1} (A_{2})$ 
by the ${\rm PSp}(2, \mathbb{C})$-action on $\mathbb{C}P^{3}$ 
by Proposition \ref{one-to-one curve}, 
\item
the action generated by $j, k \in {\rm Sp}(1)$.
\end{itemize}
\end{prop}

\address{Graduate School of Mathematical Sciences, University of Tokyo
3-8-1 Komaba, Meguro, Tokyo 153-8914, Japan}
{kkawai@ms.u-tokyo.ac.jp}

\end{document}